\documentclass[a4paper,reqno]{amsart}

\usepackage{a4wide}
\usepackage{graphicx,color}

\usepackage[english]{babel}

\usepackage{amsmath}
\usepackage{amssymb}
\usepackage{amsthm}
\usepackage[foot]{amsaddr}
\usepackage{cite}
\usepackage{centernot}
\usepackage{mathtools}
\usepackage{bm}\newcommand{\mat}[1]{\bm{\mathsf{#1}}}

\usepackage{enumerate}
\usepackage{units}
\usepackage{subfigure}
\usepackage{setspace}
\usepackage{wrapfig}
\usepackage{cutwin}
\usepackage[export]{adjustbox}
\usepackage{url}
\usepackage{nicefrac}
\usepackage{xcolor}

\numberwithin{equation}{section}

\everymath{\displaystyle}

    \newcommand{\Id}{\text{Id}}
    
    \newcommand{\R}{\mathbb{R}}

    \newcommand{\N}{\mathbb{N}}
    
    \newcommand{\la}{\langle}
    \newcommand{\ra}{\rangle}
    \newcommand{\grad}{\nabla}

    \newcommand{\hoz}{\spcH^1_0(\Omega)}
    
    \newcommand{\into}{\int_{\Omega}}

    \newcommand{\htto}{\spcH^{2}_{\bm{\mathsf{\beta}}}(\Omega)}
    \newcommand{\htt}{\spcH^{2}_{\bm{\mathsf{\beta}}}}
    \newcommand{\T}{\mathcal{T}_h}
    \newcommand{\mesh}{\mathcal{T}}
    \newcommand{\Pb}{\Phi_{\bm{\mathsf{\beta}}}}

    \newcommand{\Ltwo}{\spcL^2(\Omega)}

    \newcommand{\spcL}{\mathrm{L}}
    \newcommand{\spcH}{\mathrm{H}}

\newcommand{\operator}[1]{\mathsf{#1}}
\newcommand{\B}{\operator{B}}

\newcommand{\E}{\operator{E}}

\renewcommand{\L}{\mathrm{L}} 
\renewcommand{\H}{\mathrm{H}}

\newcommand{\VV}{\spc{V}}
\newcommand{\WW}{\spc{W}}

\renewcommand{\d}{\mathrm{d}}
\newcommand{\ds}{\,\d s}
\newcommand{\dt}{\,\d t}
\newcommand{\dx}{\,\d\x}

\newcommand{\x}{\bm{\mathsf{x}}}
\newcommand{\spc}[1]{\mathbb{#1}}

\newcommand{\unn}{u_N^{n^\star}}
\renewcommand{\u}[1]{\bm{\mathfrak{u}}_{N_{#1}}}

\newcommand{\uu}{\bm{\mathfrak{u}}}

\newcommand{\un}{u^\star_N}
\newcommand{\an}{\left|\into (f(\x,u)-f(\x,\un))(u-\un) \dx \right|}
\newcommand{\ani}{\left|\into (f(\x,u)-f(\x,\un))(u-I_hu) \dx\right|}
\newcommand{\bn}{\|\grad(u-\un)\|_{\L^2(\Omega)}}
\newcommand{\bni}{\|\grad(u-I_hu)\|_{\L^2(\Omega)}}
\newcommand{\dn}{\la \E'(u)-\E'(\un),u-\un \ra}
\newcommand{\dni}{\la \E'(u)-\E'(\un),u-I_hu \ra}
\newcommand{\PP}{\Upsilon_u}
\newcommand{\DD}{\mathbb{D}}

\newcommand{\CP}{C_{\mathrm{P}}}

\DeclareMathOperator*{\esssup}{ess\,sup}

\newcommand{\KK}{\mathcal{K}_u}
\newcommand{\enpar}{e_N^{\parallel}}
\newcommand{\enperp}{e_N^{\perp}}
    
\newcommand{\GM}{C_{\mathrm{I}}}
\newcommand{\ER}{C_{\mathrm{ER}}}

\newcommand{\limab}{\varsigma}

\newtheorem{theorem}{Theorem}[section]
\newtheorem{corollary}[theorem]{Corollary}
\newtheorem{lemma}[theorem]{Lemma}
\newtheorem{proposition}[theorem]{Proposition}

\theoremstyle{definition}
\newtheorem{definition}[theorem]{Definition}

\newtheorem{remark}{Remark}

\newtheorem{assumption}[theorem]{Assumption}

\begin{document}

\author{Florian Spicher \and Thomas P.~Wihler}
\address{Mathematics Institute, University of Bern, CH-3012 Switzerland}

\title[Iterative FEM for non-monotone semilinear elliptic PDEs]{Iterative finite element approximation of non-monotone semilinear diffusion-reaction equations}



\thanks{The authors acknowledge the financial support of the Swiss National Science Foundation (SNSF), Grant No.~$200021\underline{\phantom{a}}212868$.}

\begin{abstract}
We study iterative linearized finite element methods for the numerical approximation of semilinear elliptic boundary value problems with nonlinear reaction terms of asymptotically linear growth. Our approach reaches considerably beyond the classical theory by allowing for nonlinearities that are \emph{not} necessarily monotone. We investigate a stabilized implicit-explicit (IMEX) iteration combined with first-order finite element discretizations on graded meshes that are able to resolve corner singularities in polygonal domains. To account for the reduced regularity of the analytical solutions, we establish regularity estimates in corner-weighted Sobolev spaces. The principal novelty of the paper is a non-standard analysis of the approximation properties of the discrete Galerkin limits resulting from the iterative process, which yields optimal \emph{a priori} convergence estimates. Numerical experiments underline the theoretical findings.
\end{abstract}

\maketitle

\section{Introduction}
      
    On a bounded polygonal domain $\Omega \subset\R^2$, we seek numerical approximations of solutions $u\in\hoz$ of the semilinear elliptic boundary value problem
    \begin{subequations}\label{Pb:main PDE}
     \begin{alignat}{2}
            -\Delta u &=  f(\cdot,u)&\qquad&  \text{in } \Omega \label{Pb:main PDEa}\\
            u&=0&& \text{on }\partial \Omega \label{Pb:main PDEb}.
    \end{alignat}
    \end{subequations}
    The focus of this paper is on nonlinear reaction terms $f:\,\Omega\times\R\to\R$ for which \emph{no monotonicity assumptions} are imposed. This implicates substantial challenges for the numerical analysis of the underlying problem, as standard approximation techniques (such as C\'ea's lemma in the finite element setting) or contractive iterative schemes (based on Banach's fixed-point theorem) are no longer readily applicable; for a priori error estimates in the monotone case, we refer to the recent works~\cite{HHSW:26,SpicherWihler2025,Vexler:25}. Nevertheless, if $f$ exhibits asymptotically linear growth in the second argument in the sense that
    \begin{equation}\label{eq:ling}
    |f(\cdot,s)|\le C(1+|s|)\quad\text{as}\quad|s|\to\infty\qquad\text{a.e. in }\Omega,
    \end{equation}
    then the introduction of a suitable $\L^2$-stabilization into~\eqref{Pb:main PDEa} restores some key monotonicity properties. This crucial observation can also be exploited for numerical purposes: in fact, for a suitable initial guess, and a sufficiently small (fixed) parameter $\Delta t>0$, it can be shown that the iterative procedure given by
    \begin{subequations}\label{eq:discretization}
    \begin{alignat}{2}
            - \Delta u^{n+1}+\frac{1}{\Delta t}u^{n+1}  &= \frac{1}{\Delta t}u^n + f(\cdot,u^n)&\qquad& \text{in } \Omega\\
            u^{n+1} &=0&\qquad & \text{on } \partial \Omega,
    \end{alignat}
    \end{subequations}
    for $n\geq 0$, generates a sequence $\{u_n\}_{n\ge 0}$ that converges weakly to a solution $u\in\hoz$ of~\eqref{Pb:main PDEa}; we refer the interested reader to the method of sub- and supersolutions in the elliptic PDE literature (see, e.g., \cite[\S9.3]{Evans2010}). Alternatively, the above iteration can be interpreted as an implicit-explicit (IMEX) time-discretization scheme of the semilinear parabolic evolution equation 
    \begin{equation*}
        \partial_t v(\x,t) - \Delta v(\x,t) = f(\x,v(\x,t))\qquad (\x,t)\in\Omega\times(0,\infty),
    \end{equation*}
    for a time step $\Delta t>0$. 
    
    From a practical point of view, it is important to note that~\eqref{eq:discretization} amounts to a \emph{linear solve} at each iteration step, making the resulting procedure particularly appealing for numerical approximation once it is discretized in space, for instance, by finite element methods. Indeed, in the recent work \cite{Wih}, under uniformly bounded derivative conditions on the reaction term~$f$ (as in~\eqref{eq:ling} and to be specified later on in \S\ref{sc:lg}), the stabilized IMEX iteration \eqref{eq:discretization} has been shown to yield (weak) convergence in closed subspaces (including, in particular, finite-dimensional Galerkin subspaces), and that the resulting limits tend to a weak solution of~\eqref{Pb:main PDE} on a hierarchical $\hoz$-dense sequence of subspaces; cf.~\cite[Cor.~3.12]{Wih} and~\cite[Thm.~4.2]{Wih}, respectively. 
    
    A key aspect of the work~\cite{Wih} is the introduction of a novel \emph{energy-based} analytical framework, which allows the iteration~\eqref{eq:discretization} to be interpreted as a local minimization process for an energy functional $\E:\hoz\to\R$ associated with the PDE~\eqref{Pb:main PDEa}; see~\S\ref{sc:energy functional} below. Crucially, by taking into account the underlying energy of the problem (which is assumed to be weakly coercive) rather than focusing solely on the PDE itself, additional structure that proves decisive for the numerical analysis is acquired.

    In the present work, we continue to pursue the promising energy-based approach introduced in~\cite{Wih}, and substantially extend the existing theory by developing a complete \emph{a priori} error analysis for finite element approximations of~\eqref{Pb:main PDE}. In particular, our objective is to establish optimal convergence rates for $\mathbb{P}_1$-finite element discretizations in polygonal domains. To this end, we develop non-standard analytical techniques that go well beyond the classical theory for coercive or monotone operators. Moreover, we admit the presence of elliptic corner singularities necessitating suitably graded meshes. The main contributions of this paper are twofold:
\begin{enumerate}[1.]
\item
We establish a regularity result for~\eqref{Pb:main PDE} in a Kondrat'ev-type framework~\cite{kondrat1967boundary}, which accounts for possible corner singularities of the solution. More precisely, by exploiting the regularity theory for linear elliptic problems developed in~\cite{Babuska1986,Babuska1979}, we show that the solution of~\eqref{Pb:main PDE} belongs to a scale of corner-weighted Sobolev spaces. These spaces exhibit classical $\spcH^2$-regularity away from the corners of~$\Omega$, while permitting singular behavior at the vertices. This regularity result naturally motivates the use of graded meshes near the corners of~$\Omega$, originally introduced in~\cite{Babuska1979} (see also~\cite{SpicherWihler2025} for monotone semilinear PDEs), which recover optimal approximation rates despite the reduced regularity.

\item
The central idea of our error analysis is to distinguish several asymptotic regimes, which we characterize by a finite parameter $\limab\in\R$, see \S\ref{sc:OCR general case} (in particular~\eqref{eq:def ell}). The sign of $\limab$ reflects the asymptotic balance between the diffusive and reactive parts  in the Gâteaux derivative of the associated energy functional. When $\limab<0$, we obtain a quasi-optimal error estimate, analogous to the classical finite element theory. Similarly, in the borderline case $\limab=0$, the nonlinear reaction becomes asymptotically negligible. The most challenging regime is $\limab>0$, where the classical machinery fails. In this case, we apply a semilinear Aubin-Nitsche-type argument, inspired by~\cite{Hardering2017}, for which we prove a new existence and stability result (in weighted Sobolev spaces) using Fredholm theory. Then, our main theorem, Thm.~\ref{thm:OCR Galerkin}, establishes optimal convergence rates for the proposed IMEX iteration based on $\mathbb{P}_1$-finite element spaces with suitable local mesh grading. In particular, we demonstrate that numerical approximations of~\eqref{Pb:main PDE} of any prescribed accuracy can be computed in finitely many steps; moreover, in the contractive regime, the required number of iterations scales only logarithmically with respect to the number of degrees of freedom of the underlying finite element space, while preserving the optimal convergence rate.
\end{enumerate}
 
\subsubsection*{Outline}
    We begin by fixing the variational setting, and revisit a few instrumental tools from \cite{Wih} in \S\ref{sc:framework}; in particular we introduce parametrized norms in \S\ref{sc:function space}, discuss precisely the linear growth condition on the nonlinear reaction term~$f$ in \S\ref{sc:lg}, and present the energy functional with boundedness and Lipschitz properties related to its first variation in \S\ref{sc:energy functional}. The regularity of~\eqref{Pb:main PDE} in scales of weighted Sobolev spaces will be discussed in~\S\ref{sc:reg}. Furthermore, in \S\ref{sc:scheme} we focus on the IMEX iteration~\eqref{eq:discretization}, and on Galerkin approximations of the solution of \eqref{Pb:main PDE} in closed subspaces. In addition, \S\ref{sc: convergence} proves weak convergence for sequences of exact Galerkin limits on an increasing family of finite-element subspaces to a weak solution of \eqref{Pb:main PDE} (see Prop.~\ref{prop: u critical Galerkin}). In \S\ref{sc:ocr} we derive a priori error estimates: we first prove general quasi-optimal bounds and an optimal convergence result for the Galerkin limits in special settings treated in \S\ref{ssc:OCR special case}--\S\ref{ssc:OCR monotone}; by a specifically tailored Aubin-Nitsche duality argument, combined with the graded-mesh interpolation properties, we then derive the rate in the general setting under sharp assumptions (Thm.~\ref{thm:OCR Galerkin}). Additionally, \S\ref{ssc:OCR grad} treats the existence of strongly convergent subsequence of the Galerkin limits (Cor.~\ref{cor:strong conv}). We conclude the theoretical part of this work with \S\ref{sc:approx}, which discusses the attainability of the optimal rate with a finite number of iterations: under the additional assumption that the iterates converge to a targeted Galerkin (Rmk.~\ref{rmk:OCR approx}),
    the iteration error can be made negligible relative to the approximation error on any fixed discrete space $\VV_N$, so that a prescribed tolerance is reached after finitely many steps. 
    Finally, we test our theoretical findings through a series of numerical experiments in \S\ref{sc:numerics}, which highlight the sharpness of our assumptions, and we draw some conclusions in \S\ref{sc:conclusion}.

\section{Energy-based framework and regularity}\label{sc:framework}
    
In this section, we briefly review an appropriate variational setting, including a precise growth condition on the nonlinearity~$f$ in~\eqref{Pb:main PDEa}, which has recently been proposed in~\cite{Wih}, and that serves as the basis of an a priori energy-oriented numerical analysis to be developed in this paper. In addition, we establish the regularity of~\eqref{Pb:main PDE} in polygonal domains~$\Omega\subset\R^2$.
    
\subsection{Function spaces}\label{sc:function space}
    Let $\VV:= \hoz$ be the Sobolev space of all functions in $\spcH^1(\Omega)$ with vanishing boundary trace along $\partial\Omega$, equipped with the $\spcH^1$-seminorm $\|\cdot\|_{\VV}:=\|\nabla(\cdot)\|_{\L^2(\Omega)}$. We also consider the dual space $\VV'$, and the duality pairing $\la \cdot, \cdot \ra$ in $\VV' \times \VV$. In accordance with the left-hand operator in~\eqref{eq:discretization}, the space $\VV$ is endowed with a parametrized family of norms $\| \cdot \|_\lambda$, where, for $\lambda>0$, we define
    \begin{equation}\label{eq:lambda norm}
        \|v\|_\lambda := \left(\|v\|_{\L^2(\Omega)}^2+\lambda\|\grad v\|_{\L^2(\Omega)}^2\right)^{\nicefrac{1}{2}}, \qquad v\in \VV,
    \end{equation}
    with $\|\cdot\|_{\L^2(\Omega)}$ denoting the standard $\L^2$-norm. Using the Poincar\'e inequality,
    \begin{equation}\label{eq:Poincare L2}
        \|v\|_{\L^2(\Omega)}^2 \leq \CP \|\grad v\|_{\L^2(\Omega)}^2 \qquad \forall v \in \VV,
    \end{equation}
    with a constant $\CP > 0$ depending only on $\Omega$, we recall the following result for the norm $\|\cdot\|_\lambda$. Referring to~\cite[Lem.~2.1]{Wih}, for any $v\in\VV$ and $\lambda>0$, it holds the bound
    \begin{subequations}
        \begin{equation}\label{eq:Poincare lambda}
            \|v\|_{\L^2(\Omega)}^2 \leq \theta(\lambda) \|v\|_\lambda^2,
        \end{equation}
        with $0 < \theta(\lambda) < 1$ given by 
        \begin{equation}\label{def:beta(lambda)}
            \theta(\lambda):= \frac{\CP}{\CP+\lambda},
        \end{equation}
        where $\CP>0$ is the Poincar\'e constant from \eqref{eq:Poincare L2}.
    \end{subequations}

    \subsection{Asymptotic linear growth}\label{sc:lg}
    Following the approach~\cite{Wih}, we consider nonlinear reaction terms $f$ in~\eqref{Pb:main PDEa} that are linearly growing, cf.~\eqref{eq:ling}, with sharp upper and lower bounds on the partial derivative~$f_u$, see~\eqref{eq:f' bounds} below. More precisely, the following conditions will be imposed throughout.

    \begin{assumption}\label{Ass: f}
        The function $f: \Omega \times \R \rightarrow \R$ satisfies
        \begin{enumerate}[(i)]
            \item \label{Ass: f L2} $f(\cdot,s)\in \L^2(\Omega)$ for all $s\in \R$;
            \item \label{Ass: fu existence} $f$ is partially (but not necessarily continuously) differentiable in the second variable; we write $f_u \equiv \partial_uf$;
            \item \label{Ass: fu bound} There exists a constant $\rho>0$ such that the set
            \begin{equation}\label{eq: Lambda(rho)}
                \Lambda_f(\rho) := \{\lambda>0 \; : \; \sigma_f(\lambda) < \rho + \nicefrac{1}{\lambda}\}
            \end{equation}
            is nonempty, where we let
            \begin{equation}\label{eq: sigma(lambda)}
                \sigma_f(\lambda) := \esssup_{\x \in \Omega}\sup_{u \in \R} \; \left| f_u(\x,u) + \frac{1}{\lambda} \right|, \qquad \lambda >0.
            \end{equation}
        \end{enumerate}
    \end{assumption}

    Starting from \S\ref{sc:ocr}, in order to derive our convergence rate results, we will refine {the above} condition \eqref{Ass: fu existence} by requiring $f$ to be \textit{continuously} differentiable, {however without any monotonicity requirement} (which is the key novelty {of our work}). If Asm.~\ref{Ass: f} is fulfilled, then $f$ is a Carathéodory function, i.e., 
    \begin{subequations}\label{eq:Caratheodory}
    \begin{alignat}{2}
            \Omega\ni\x &\mapsto f(\x,s) \text{ is measurable for every } s\in \R,\\
            \intertext{and}
            \R\ni s &\mapsto f(\x,s) \text{ is continuous for almost every } \x\in \Omega.
    \end{alignat}
    \end{subequations}
    Besides, for $\lambda \in \Lambda_f(\rho)$, the function $g_\lambda$ defined by 
    \begin{equation}\label{eq: g_lambda}
        g_\lambda(\x,u) := f(\x,u) + \lambda^{-1} u, \qquad (\x,u) \in \Omega \times \R,
    \end{equation}
    satisfies the uniform Lipschitz continuity bound
    \begin{equation}\label{eq: g Lipschitz}
        |g_\lambda(\x,u)-g_\lambda(\x,v)|\leq \sigma_f(\lambda)|u-v| \qquad \forall u,v\in \R,
    \end{equation}
    for almost every $\x \in \Omega$, where $\sigma_f(\lambda)<\rho+\nicefrac{1}{\lambda}$, cf. \eqref{eq: Lambda(rho)}.

    Furthermore, we specify bounds for the lower and upper values on the partial derivative $\partial_u f$, which are essentially equivalent to \eqref{eq: Lambda(rho)} and \eqref{eq: sigma(lambda)}: 
    \begin{equation}\label{eq:f' bounds}
        -\rho-\mu_f \leq f_u(\x,u) <\rho \qquad \text{for a.e. } \x\in\Omega \text{ and for all } u\in \R,
    \end{equation}
    where
    \begin{equation}\label{eq:def mu}
        \mu_f:= \begin{cases}
            2(\sup\Lambda_f(\rho))^{-1}, & \text{if } \sup\Lambda_f(\rho) < \infty;\\
            0, & \text{otherwise.}
        \end{cases}
    \end{equation}
    In particular, the \emph{undershooting coefficient} $\mu_f$ introduced in \cite{Wih} allows for a precise estimate on the pseudo-time parameter $\Delta t>0$ present in \eqref{eq:discretization} to reach steady-state stability.

\subsection{Variational framework}\label{sc:energy functional}
    In the paper~\cite{Wih}, a new energy-based analysis for the boundary value problem~\eqref{Pb:main PDE}, which, in particular, allows to study the iterative scheme~\eqref{eq:discretization} from the viewpoint of local energy minimization, has been presented. Indeed, the PDE~\eqref{Pb:main PDE} is associated with the
    energy functional $\E: \VV \rightarrow \R$ defined by
    \begin{equation}\label{eq:energy func}
        \E(u) := \frac{1}{2}\into |\grad u|^2 \; \dx - \into F(\x,u) \; \dx,
    \end{equation}
    where, for $\x \in \Omega$, we let
    \begin{equation}\label{eq:F}
        F(\x,t) := \int_0^t f(\x,s) \; \d s, \qquad t \in \R.
    \end{equation}
    Under Asm.~\ref{Ass: f}, for $\rho>0$ and $\lambda \in \Lambda_f(\rho)$, the energy functional from \eqref{eq:energy func} is well-defined on $\VV$.
    Furthermore, for any $u \in \VV$, the G\^ateaux derivative of $\E$ {at $u\in\VV$} is given by
    \begin{equation}\label{eq:1st variation}
       {\E'(u)\in\VV':\qquad}\la \E'(u),v\ra = \into \left(\grad u \cdot \grad v - f(\x,u)v \right)\dx \qquad \forall v \in \VV.
    \end{equation}
    {Within this setting, the Euler-Lagrange formulation of the minimization} problem
    \begin{equation*}
        u\in \VV: \qquad \E(u)=\min_{v\in\VV} \E(v),
    \end{equation*}
    {is expressed in weak form by}
    \begin{equation}\label{Pb:weak form}
        u\in \VV: \qquad \la \E'(u),v\ra = \into (\grad u \cdot \grad v - f(\x,u)v ) \dx =0 \qquad \forall v \in \VV.
    \end{equation}
    In particular, any critical point (and, specifically, any minimizer) of $\E$ in $\VV$ {is} a solution of \eqref{Pb:weak form}, which is equivalent to being a weak solution of \eqref{Pb:main PDE}.

    \begin{remark}\label{rmk:special weak coercivity}
        In the special case where $\rho\leq \nicefrac{1}{\CP}$, with $\CP$ the Poincar\'e constant from \eqref{eq:Poincare L2}, the energy functional $\E$ is \textit{weakly coercive}, i.e.,
        \begin{equation*}
            \E(v) \rightarrow + \infty \qquad \text{whenever} \qquad \|v\|_\lambda \rightarrow \infty,
        \end{equation*}
        for any fixed $\lambda >0$.
    \end{remark}

    We now collect a few instrumental results for the subsequent analysis.

    \begin{lemma}[Lipschitz continuity of $f$]\label{lemma: f Lipschitz}
        Suppose Asm.~\ref{Ass: f} is satisfied for some $\rho>0$, and that $f_u$ is measurable in the first and pointwise continuous in the second argument. Then, it holds
        \begin{equation}\label{eq: f Lipschitz}
            \left|\into (f(\x,u)-f(\x,v))w \; \dx\right| \leq (\rho+\mu_f)\|u-v\|_{\L^2(\Omega)}\|w\|_{\L^2(\Omega)} \qquad \forall u,v,w\in \VV.
        \end{equation}
    \end{lemma}
    \begin{proof}
    {For $u,v\in\VV$, and $t\in[0,1]$, define $\Gamma_t:=tu+(1-t)v$. Then, by the main theorem of calculus, we have
        \begin{align*}
            \into (f(\x,u)-f(\x,v))w \; \dx
            &=\into \left(\int_0^1\frac{\d}{\dt}f(\x, \Gamma_t)\;\dt\right)w \; \dx
            = \into \left(\int_0^1 f_u(\x,\Gamma_t) \; \dt \right) (u-v)w \; \dx.
        \end{align*} 
    Hence, using the Cauchy-Schwarz inequality, we obtain}
        \begin{align*}
            \left|\into (f(\x,u)-f(\x,v))w \; \dx\right| 
            &\leq \|u-v\|_{\L^2(\Omega)}\|w\|_{\L^2(\Omega)}\int_0^1\|f_u(\cdot,\Gamma_t)\|_{\L^\infty(\Omega)} \; \dt \\
            &\leq (\rho+\mu_f)\|u-v\|_{\L^2(\Omega)}\|w\|_{\L^2(\Omega)},
        \end{align*}
        where we have used \eqref{eq:f' bounds} to derive the last inequality.
    \end{proof}

    The ensuing result is borrowed from~\cite[Lem.~2.6]{Wih}.

    \begin{lemma}[Lipschitz Continuity of $\E'$]\label{lemma: E' Lipschitz}
        {Under} Asm.~\ref{Ass: f} for some $\rho>0$, and $\lambda\in \Lambda_f(\rho)$, cf. \eqref{eq: Lambda(rho)}, {the energy derivative} $\E':\VV \rightarrow \VV'$ is (uniformly) Lipschitz continuous in the sense that
        \begin{subequations}\label{eq: E' Lipschitz}
        \begin{equation}
            |\la \E'(u)-\E'(v),w\ra| \leq L_{\E'}(\lambda) \|u-v\|_\lambda\|w\|_\lambda \qquad \forall u,v,w \in \VV,
        \end{equation}
        where
        \begin{equation}
            L_{\E'}(\lambda) := \lambda^{-1}+\theta(\lambda)\sigma_f(\lambda),
        \end{equation}
        \end{subequations}
        with $\theta(\lambda)$ from \eqref{def:beta(lambda)}.
    \end{lemma}

\subsection{Regularity}\label{sc:reg}
In situations where the polygonal domain $\Omega \subset \mathbb{R}^2$ contains non-convex corners, it is well-known that the inverse Laplacian $(-\Delta)^{-1}$ does not exhibit full elliptic regularity. Consequently, solutions of Poisson-type boundary value problems, including~\eqref{Pb:main PDE}, are typically found to be $\spcH^2$ away from the boundary~$\partial\Omega$;  however, $\spcH^2$-regularity does not extend uniformly to the corners of $\Omega$ in general. We will address this issue within the framework of corner-weighted Sobolev spaces; see, e.g., \cite{Babuska1986,Babuska1979,Schwab1998}.
    To this end, we define the weight function
    \begin{equation}\label{def:Pb}
        \Pb(\bm{\mathsf{x}})= \prod_{j=1}^m \text{dist}(\bm{\mathsf{x}},\bm{\mathsf{c}}_j)^{\beta_j},
    \end{equation}
    which involves the Euclidean distances from a point $\mat x\in\overline{\Omega}$ to any of the $m\geq 3$ corners of the polygon $\Omega$, denoted by $\mat c_1,\ldots,\mat c_m$, and associated exponents $\bm{\mathsf{\beta}}=(\beta_1,\ldots,\beta_m)$, with $0\leq \beta_1,\ldots,\beta_m <1$. We let 
    $\spcL_{\bm{\mathsf{\beta}}}(\Omega)$ be the space of all measurable functions $v$ on $\Omega$ for which the norm
    $
    \|v\|_{\spcL_{\bm{\mathsf{\beta}}}(\Omega)}
    :=\left\|\,\Phi_{\bm{\mathsf{\beta}}} v\,\right\|_{\L^2(\Omega)}
    $
    is finite. Then, we define the weighted $\spcH^2$-norm
     \begin{equation*}
            \|v\|^2_{\spcH^{2}_{\bm{\mathsf{\beta}}}(\Omega)} := 
            \|v\|^2_{\spcH^{1}(\Omega)}+\sum_{\alpha_1+\alpha_2=2}\|\partial_{x_1}^{\alpha_1}\partial_{x_2}^{\alpha_2}v\|^2_{\spcL_{\bm{\mathsf{\beta}}}(\Omega)},
        \end{equation*}
          as well as the associated weighted Sobolev space
        \begin{equation*}
            \spcH^{2}_{\bm{\mathsf{\beta}}}(\Omega) := \big\{v \in \spcH^1(\Omega) \; : \; \|v\|_{\spcH^{2}_{\bm{\mathsf{\beta}}}(\Omega)} < \infty\big\}.
        \end{equation*}
    We notice the continuous embedding 
    \begin{equation}\label{eq:emb}
    \htto \hookrightarrow \mathrm{C}^0(\overline{\Omega})\,;
    \end{equation}
    see, e.g., \cite[Eq.~(2.2)]{Babuska1979}.

    \begin{theorem}[Regularity of weak solution]\label{thm:beta}
        Suppose that the polygon $\Omega$ is non-degenerate, i.e., all interior angles at the corners $\mat c_1,\ldots,\mat c_m$, which we signify by $\omega_1,\ldots,\omega_m$, fulfill $\omega_j\in(0,2\pi)$ for each $j=1,\ldots,m$, and suppose Asm.~\ref{Ass: f} is satisfied for some $\rho>0$. If the weight exponents $\beta_1,\ldots,\beta_m$ in~\eqref{def:Pb} obey the bound
        \begin{equation}\label{eq:beta LB}
          1>\beta_j >  1-\nicefrac{\pi}{\omega_j}
        \end{equation}
        then any weak solution of~\eqref{Pb:main PDE} belongs to $\VV \cap \htto$.
    \end{theorem}
    \begin{proof}
        We observe that $f(\cdot,u)\in  \spcL_{\bm{\mathsf{\beta}}}(\Omega)$ for $u\in \VV$. Indeed, by the fundamental theorem of calculus and the bounds \eqref{eq:f' bounds} it holds 
        \begin{equation*}
            |f(\x,u)| \leq |f(\x,0)| + \int_0^1 |f_u(\x,su)|\,|u| \; \ds \leq |f(\x,0)| + (\rho+\mu_f)|u|,
        \end{equation*}
        for almost every $\x\in \Omega$, which implies that $f(\cdot,u)\in \spcL^2(\Omega)\subset \spcL_{\bm{\mathsf{\beta}}}(\Omega)$
        thanks to Asm.~\ref{Ass: f}\eqref{Ass: f L2}. The proof is then completed following the lines of \cite[Thm.~2.6]{SpicherWihler2025}.        
    \end{proof}
    
\section{Iterative approximation}\label{sc:scheme}

We will now turn our attention to the pseudo-time iteration scheme~\eqref{eq:discretization}, and will discuss its convergence, with a particular focus on hierarchical Galerkin discretizations. Proofs can be found in \cite[\S2 and \S3]{Wih}.

\subsection{Iteration on the full space}
Consider a fixed time-step $\Delta t > 0$, and an appropriate initial guess $u^0 \in \L^2(\Omega)$. Then, for a given $u^n$, $n \geq 0$, the weak form of~\eqref{eq:discretization} is to seek $u^{n+1} \in \VV$ such that
    \begin{equation}\label{eq: iteration scheme}
        \frac{1}{\Delta t} \into (u^{n+1}-u^{n})v \; \dx + \into \grad u^{n+1}\cdot  \grad v \; \dx = \into f(\x,u^{n})v\; \dx \qquad \forall v \in \VV.
    \end{equation}
    It is important to note that the iteration \eqref{eq: iteration scheme}, for a given $u^n$, represents a \textit{linear} problem for $u^{n+1}$, meaning {that} it can be interpreted as an \textit{iterative linearization} of \eqref{Pb:main PDE}. In fact, it can be expressed equivalently as
    \begin{equation}\label{eq:forms iteration scheme}
        \B_{\Delta t}(u^{n+1},v) = \ell_{\Delta t}(u^n;v) \qquad \forall v \in \VV,
    \end{equation}
    where, for $\lambda>0$, we define the \textit{bilinear form} $\B_\lambda:\VV\times \VV \rightarrow \R$ by
    \begin{equation}\label{eq:B form}
        \B_\lambda(u,v) := \into \lambda \grad u \cdot \grad v + uv \; \dx \qquad u,v \in \VV,
    \end{equation}
    and, for a given $y \in \L^2(\Omega)$, the \textit{linear form} $\ell_\lambda(y;\cdot): \VV \rightarrow \R$ by
    \begin{equation}\label{eq:l form}
        \ell_\lambda(y;v) := \lambda \into g_\lambda(\x,y)v \; \dx \qquad v \in \VV,
    \end{equation}
    cf.~\eqref{eq: g_lambda}. {Moreover,} recalling \eqref{eq:1st variation}, for {any} $u\in \VV$, we notice the identity
    \begin{equation*}
        \la \E'(u),v\ra = \lambda^{-1} (\B_\lambda(u,v)-\ell_\lambda(u;v)) \qquad \forall v\in \VV.
    \end{equation*}
    
    The bilinear form $\B_\lambda$ from \eqref{eq:B form} defines an inner product on $\VV \times \VV$ {that} induces the norm from \eqref{eq:lambda norm}. Specifically, for a fixed $\lambda > 0$, this norm is equivalent to the standard $\H^1$-norm, meaning that the space $\VV$ endowed with \eqref{eq:B form} is a Hilbert space. Furthermore, if Asm.~\ref{Ass: f} holds for a fixed $\Delta t \in \Lambda_f(\rho)$, cf. \eqref{eq: Lambda(rho)}, then the linear form $\ell_{\Delta t}(u^n;\cdot)$ is bounded, and hence the iteration \eqref{eq:forms iteration scheme} is well-defined for all $n \geq 0$ and any initial guess $u^0 \in \L^2(\Omega)$; see~\cite[Lem.~3.1]{Wih} and \cite[Prop.~3.2]{Wih}, respectively.

\subsection{Iteration in closed subspaces}\label{sc:iteration subspace}
    Let now $\WW \subset \VV$ be a closed subspace of $\VV$ (e.g., a finite-dimensional Galerkin subspace, or $\VV$ itself). We now restrict the weak formulation \eqref{Pb:weak form} to $\WW$, namely:
    \begin{equation}\label{Pb:weak form W}
        u^\star \in \WW: \qquad \into \grad u^\star \cdot \grad v \; \dx = \into f(\x,u^\star)v\; \dx \qquad \forall v \in \WW.
    \end{equation}
    This corresponds to the Euler-Lagrange equation for critical points of the energy functional $\E$ from \eqref{eq:energy func} on the subspace $\WW$. In particular, we have
    \begin{equation}\label{eq:Galerkin orthogonality}
        u^\star\in\WW: \qquad \la \E'(u^\star),v\ra =0 \qquad \forall v\in\WW,
    \end{equation}
    where the first variation $\E'$ is defined in \eqref{eq:1st variation}.
    
    In accordance with \eqref{eq:forms iteration scheme}, starting with some initial guess $u^0 \in \L^2(\Omega)$, we define the \textit{iterative linearization} scheme
    \begin{equation}\label{eq:forms iteration scheme W}
        u^{n+1} \in \WW: \qquad \B_{\Delta t}(u^{n+1},v)=\ell_{\Delta t}(u^n;v) \qquad \forall v \in \WW,
    \end{equation}
    {for $n\ge 0$.} With arguments {similar} to \cite[Prop.~3.2]{Wih}, {if $\Delta t \in \Lambda_f(\rho)$, then} the {problem} \eqref{eq:forms iteration scheme W} is well-posed for all $n \geq 0$. 

It can be shown that the pseudo-time iteration scheme~\eqref{eq:forms iteration scheme W} converges provided that the time step $\Delta t>0$ is chosen sufficiently small, and that the energy functional $\E$ from~\eqref{eq:energy func} is weakly coercive. More precisely, we impose the following conditions.

    \begin{assumption}[Variational stability]\label{Ass:convergence}
        Let $\rho>0$ such that Asm.~\ref{Ass: f}\eqref{Ass: fu bound} is fulfilled, and let $\Delta t \in \Lambda_f(\rho)$. {We choose $\Delta t>0$ such that
        \begin{enumerate}[(i)]
            \item \label{Ass: delta_t bound}$\nicefrac{1}{\Delta t}>\nicefrac{1}{2}\max\{\rho+\mu_f-\nicefrac{1}{\CP},0\}$.
        \end{enumerate}
        Furthermore, we suppose that}
        \begin{enumerate}[(i)]
        \setcounter{enumi}{1}
        \item \label{Ass: E weak coercivity} $\E$ is weakly coercive, i.e., for every $v\in \VV$, we have $\E(v)\rightarrow +\infty$ whenever $\|v\|_{\Delta t}\rightarrow \infty$.
        \end{enumerate}
    \end{assumption}
    
    The above assumption is a stability condition for {the iteration}~\eqref{eq:forms iteration scheme W}. 
    In particular, Asm.~\ref{Ass:convergence}\eqref{Ass: delta_t bound} implies that the corresponding sequence of energies $\{\E(u^n)\}_{n\ge 0}$ is monotone decreasing, see \cite[Prop.~3.4]{Wih}, which, in combination with Asm.~\ref{Ass:convergence}\eqref{Ass: E weak coercivity}, yields that the sequence $\{u^n\}_{n\geq 0}\subset \WW$ is bounded. As a consequence, the following convergence result holds, see~\cite[Thm.~3.10]{Wih} and~\cite[Cor.~3.12]{Wih}, which will be applied repeatedly in the ensuing sections.
    
    \begin{theorem}[Convergence in closed subspaces]\label{thm: convergence in closed sbspc}
        Let $\WW\subseteq \VV$ be a closed linear subspace, and let Asm.~\ref{Ass:convergence} be fulfilled. Then, for any initial guess $u^0\in\L^2(\Omega)$, the sequence $\{u^n\}_{n\geq 0}\subset \WW$ generated by the iteration \eqref{eq:forms iteration scheme W} has a subsequence $\{u^{n_k}\}_{k\geq 0}$ that converges weakly in $\WW$ and strongly in $\mathrm{L}^2(\Omega)$ to a solution  $u^\star\in \WW$ of~\eqref{Pb:weak form W}; moreover, if $\dim(\WW)<\infty$ (e.g., in the case of a {discrete} subspace), then the convergence to $u^\star$ is strong in~$\WW$.
    \end{theorem}   

\begin{remark}\label{rmk:special case}
In the special case $\rho \leq \nicefrac{1}{\CP}$, with $\CP > 0$ from \eqref{eq:Poincare L2}, we make the following observations.
\begin{enumerate}[(i)]
\item In the setting where Asm.~\ref{Ass: f}\eqref{Ass: fu bound} holds for $\rho \leq \nicefrac{1}{\CP}$, we note that Asm.~\ref{Ass:convergence}\eqref{Ass: E weak coercivity} is fulfilled according to Rmk.~\ref{rmk:special weak coercivity}. In addition, for $\Delta t\in\Lambda_f(\rho)$, it holds that $\mu_f< \nicefrac{2}{\Delta t}$, cf.~\eqref{eq:def mu}, whence we observe that
    $
     \rho+\mu_f-\nicefrac{1}{\CP}\le \mu_f< \nicefrac{2}{\Delta t}, 
    $
    and Asm.~\ref{Ass:convergence}\eqref{Ass: delta_t bound} is satisfied for any $\Delta t\in\Lambda_f(\rho)$.  Indeed, this follows from the bound
    \begin{equation*}
        |\sigma_f(\lambda_1)-\sigma_f(\lambda_2)|\leq \left|\lambda_1^{-1}-\lambda_2^{-1}\right| \qquad \forall \lambda_1,\lambda_2\in \Lambda_f(\rho),
    \end{equation*}
    which implies that, for any $\lambda\in\Lambda_f(\rho)$, there is $\varepsilon>0$ such that $\lambda+\varepsilon\in \Lambda_f(\rho)$, and we conclude that $\lambda<\sup\Lambda_f(\rho)$.     
\item Moreover, if 
    $\mu_f= 0$
    then it is even possible to let $\nicefrac{1}{\Delta t}\to0$,  allowing us to neglect the mass term on the left-hand side of the iterative scheme \eqref{eq: iteration scheme}, and thereby reducing it to the iteration
        \begin{equation}\label{eq:reduced scheme}
            \into \grad u^{n+1}\cdot \grad v \; \dx = \into f(\x,u^n)v\; \dx \qquad \forall v\in \VV.
        \end{equation}

        \item
    Finally, for any closed subspace $\WW\subset\VV$ the mapping $\mathsf{T}_{\Delta t}: \L^2(\Omega) \rightarrow \WW$ defined by
    \begin{equation}\label{def:op T}
        y\mapsto\mathsf{T}_{\Delta t}(y):\qquad
        \B_{\Delta t}(\mathsf{T}_{\Delta t}(y),v)=\ell_{\Delta t}(y;v) \qquad \forall v \in \WW,
    \end{equation}
    satisfies the stability bound
    \begin{subequations}\label{eq:stability bound}
    \begin{equation}\label{eq:stability bound1}
        \|\mathsf{T}_{\Delta t}(y)-\mathsf{T}_{\Delta t}(z)\|_{\Delta t} \leq r(\Delta t) \|y-z\|_{\Delta t}, \qquad 0<r(\Delta t)<1,
    \end{equation}
    where 
    \begin{equation}\label{eq:rt}
    r(\Delta t):=\theta(\Delta t)\sigma_f(\Delta t) \Delta t,
    \end{equation}
    \end{subequations}
    for $\Delta t \in \Lambda_f(\rho)$; see~\cite[Thm.~3.3]{Wih}. Notably, the restriction $\mathsf{T}_{\Delta t}|_{\WW}:\WW \rightarrow \WW$ acts as a contraction. These observations result in the fact that the entire sequence generated by the iterative scheme~{\eqref{eq:reduced scheme}} converges strongly to a \textit{unique} critical point of $\E$ in $\WW$ with respect to the norm \eqref{eq:lambda norm} with $\lambda=\Delta t$, independently of the initial guess $u^0 \in \L^2(\Omega)$. 

    \end{enumerate}
\end{remark}

\subsection{Convergence on sequences of hierarchical Galerkin subspaces}\label{sc: convergence}

For the purpose of this paper, the generic discrete space~$\WW$ used in the previous section~\S\ref{sc:iteration subspace} will be specified by a sequence $\{\VV_N\}_{N\ge0}$ of discrete Galerkin subspaces that will be obtained from an initial (low-dimensional) space $\VV_0\subset \VV$ by successive enrichments. More specifically, later on in \S\ref{sc:ocr}, a sequence of hierarchically refined finite element spaces will be considered. The following assumption ensures that the sequence $\{\VV_N\}_{N\geq 0}$ resolves the underlying continuous space $\VV$ in an appropriate way.

    \begin{assumption}\label{ass:Galerkin density}
         Consider a sequence of finite-dimensional subspaces 
         \[
         \VV_0 \subset \VV_1 \subset \ldots\subset\VV_N \subset\ldots\subset \VV,\qquad \dim(\VV_N)<\infty,
         \]
         that satisfies the density property
        \begin{equation*}
            \overline{\bigcup_{N\geq 0} \VV_N}=\VV,
        \end{equation*}
        where the closure refers to the $\VV$-norm, see~\S\ref{sc:function space}, which, in turn, is equivalent to $\|\cdot\|_\lambda$ for any $\lambda>0$.
    \end{assumption}
    
    The approximation of solutions for~\eqref{Pb:main PDE} is based on applying the \emph{linear} iterative procedure~\eqref{eq:forms iteration scheme W} on a given discrete (and hence closed) subspace $\WW:=\VV_N\subset \VV$, $\dim(\VV_N)<\infty$, which can be refined further in terms of a hierarchical sequence $\VV_N\subset \VV_{N+1}\subset\ldots$ in order to continuously improve the quality of the outcome. More precisely, for a given $N\ge 0$, the resulting sequence $\{u^n_N\}_{n\geq 0}\subset\VV_N$ obtained by~\eqref{eq:forms iteration scheme W} converges strongly (up to a subsequence) to an exact Galerkin solution $\un\in \VV_N$, see~Thm.~\ref{thm: convergence in closed sbspc}, i.e.,
    \begin{equation}\label{eq:un}
    \un \in \VV_N: \qquad \into \grad \un \cdot \grad v \; \dx = \into f(\x,\un)v\; \dx \qquad \forall v \in \VV_N,
    \end{equation}
    cf.~\eqref{Pb:weak form W}.
    Then, on an enriched space $\VV_{N+1}$, an initial guess $u_{N+1}^0 \in \VV_{N+1}$ for the iteration~\eqref{eq:forms iteration scheme W} on $\WW=\VV_{N+1}$ is immediately obtained through the canonical embedding of the Galerkin approximation $\un\in \VV_N$ from~\eqref{eq:un} into $\VV_{N+1}$, viz.
    \begin{equation}\label{eq:u0}
    u^0_{N+1}:=\un\in\VV_N\hookrightarrow \VV_{N+1}.
    \end{equation}

    By the same arguments as in the proof of Thm.~\ref{thm: convergence in closed sbspc}, along a sequence of hierarchical subspaces that satisfies Asm.~\ref{ass:Galerkin density}, we can extract a subsequence of the associated (limit) approximations, again denoted by $\{\un\}_{N\geq 0}$, which converges strongly to a function $u$ in $\L^2(\Omega)$ and weakly in $\VV$. Furthermore, as a consequence of the same theorem we have that $\E'(\un)\overset{\star}{\rightharpoonup} \E'(u)$, i.e.,
    \begin{equation}\label{eq:weak* convergence}
        \lim_{N\rightarrow \infty}\la \E'(\un),v \ra = \la \E'(u),v\ra \qquad \forall v \in \VV.
    \end{equation}

    \begin{proposition}\label{prop: u critical Galerkin}
        Suppose Asm.~\ref{Ass: f}, Asm.~\ref{Ass:convergence} and Asm.~\ref{ass:Galerkin density} are satisfied for some $\rho>0$ and $\Delta t\in \Lambda_f(\rho)$. The limit $u\in \VV$ of the sequence $\{\un\}_{N\geq 0}$ in \eqref{eq:weak* convergence} is a critical point of $\E$, and hence a weak solution of \eqref{Pb:main PDE}.
    \end{proposition}
    \begin{proof}
        Pick an arbitrary $v\in \VV$. Thanks to Asm.~\ref{ass:Galerkin density}, we can find a sequence $\{v_N\}_{N\geq 0}$ such that $v_N \in \VV_N$, and $v_N\to v$ strongly in $\VV$ as $N\to\infty$. This motivates the decomposition
        \begin{equation*}
            \la \E'(u), v\ra = \la \E'(\un), v_N\ra + \la \E'(\un),v-v_N\ra + \la \E'(u)-\E'(\un),v\ra.
        \end{equation*}
        The first term of the right-hand side vanishes thanks to~\eqref{eq:Galerkin orthogonality}. {In addition, from~\eqref{eq:weak* convergence}}, we know that $\E'(\un)$ is uniformly bounded in $\VV'$, and so
        \begin{equation*}
            |\la \E'(\un),v-v_N\ra| \leq \|\E'(\un)\|_{\VV'}\|v-v_N\|_\VV \rightarrow 0,
        \end{equation*}
        as $N\rightarrow\infty$. As for the remaining term, it is immediate from \eqref{eq:weak* convergence} that
        \[
            \lim\limits_{N\rightarrow \infty}\la \E'(u)-\E'(\un),v\ra = 0.
        \]
        This completes the argument.
    \end{proof}
   
    \section{Rate-optimal convergence analysis}\label{sc:ocr}
    In this section, we show optimal convergence rate estimates for the Galerkin approximations~from~\eqref{eq:un} on so-called graded meshes toward the corners of the polygon $\Omega$. To this end, for the error between a solution $u\in\VV$ of~\eqref{Pb:main PDE} and a corresponding Galerkin approximation $\un\in\VV_N$, 
    we first derive quasi-optimal estimates valid under specific assumptions, and then, more generally, using an Aubin-Nitsche-type duality argument, we derive the optimal rate
    \begin{equation}\label{eq:OCR}
        \|u-\un\|_{\Delta t} \leq C\dim(\VV_N)^{-\nicefrac{1}{2}},
    \end{equation}
    for each sufficiently large $N$, where $C>0$ is a constant independent of~$\dim(\VV_N)$.

\subsection{Finite element approximations on graded meshes}\label{sc:graded meshes}
The approximation of functions in the weighted Sobolev space $\htto$ within a finite element setting mandates the use of suitably refined meshes that are able to properly resolve possible elliptic corner singularities. To this end, we recall the family of graded meshes introduced in~\cite{Babuska1979}.

    \begin{definition}[Finite spaces on graded meshes] \label{def:triangulation type}
        Let $\bm{\mathsf{\beta}}=(\beta_1,\ldots,\beta_m)\in[0,1)^m$ be a weight vector associated with the corners $\mat c_1,\ldots,\mat c_m$ of the polygon~$\Omega$, and $\Pb$ the corresponding weight function from \eqref{def:Pb}. Then, a {regular (conforming) and shape-regular triangulation} $\T=\{T\}_{T\in\T}$ of mesh size $h=\max\{h_T:\,T\in\T\}>0$, where $h_T$ denotes the diameter of any triangle $T\in\T$, is called a \emph{graded mesh} if there exists a constant $\kappa\ge 1$ such that, for all $T\in\T$, it holds
    \begin{align*}
            \kappa^{-1}\sup_T\Pb \leq \nicefrac{h_T}{h} \leq \kappa \inf_T\Pb,&\qquad\text{if $\Pb>0$ on $\overline{T}$},
            \intertext{and}
            \kappa^{-1} \leq \frac{h_T}{h \sup_{ T}\Pb}\leq \kappa,&\qquad\text{if there is a corner $\mat c_i$ of $\Omega$ with $\mat c_i\in\overline{T}$.}
    \end{align*}
    Furthermore, for a given triangulation $\T$, we define the associated $\mathbb{P}_1$-finite element space by
    \begin{equation}\label{eq:fes}
            \mathbb{S}^{1}(\Omega,\T) := \big\{u\in \VV \; : \; u|_{T} \in \mathbb{P}_1(T)\quad \forall T \in \T\big\},
        \end{equation}
        where $\mathbb{P}_1(T)$ denotes the set of all linear polynomials on a triangle $T\in\T$; this space consists of all continuous, element-wise linear functions on the graded mesh $\T$, with zero boundary values along~$\partial \Omega$.
    \end{definition}

The family of finite-dimensional spaces $\mathbb{S}^{1}(\Omega,\T)$, which are based on graded meshes~$\T$, is able to approximate functions in $\htto$ at an optimal rate as $h\to0$, i.e., qualitatively comparable to the approximation of $\spcH^2$- functions on uniform meshes. More precisely, referring to~\cite[Lem.~4.5]{Babuska1979}, we have the interpolation bound
        \begin{equation}\label{eq:graded estimate}
            \|\grad(u-I_hu)\|_{\Ltwo}\leq \GM \dim(\mathbb{S}^{1}(\Omega,\T))^{-\nicefrac{1}{2}} \|u\|_{\htt(\Omega)},
        \end{equation}
where $I_h:\,\VV\to\mathbb{S}^{1}(\Omega,\T)$ is the standard \emph{nodal interpolant} on~$\T$, and $\GM>0$ is a constant only depending on the domain $\Omega$, on the weights $\bm\beta$, and on $\kappa$; applying the embedding~\eqref{eq:emb}, we note that
\begin{equation}\label{eq:Istab1}
\|I_hv\|_{\mathrm{C}^0(\Omega)}
\le\|v\|_{\mathrm{C}^0(\Omega)}
\le C\|v\|_{\htto}
\qquad\forall v\in\htto.
\end{equation}

Moreover, it can be seen that
\begin{equation}\label{eq:dim S1}
        \dim(\mathbb{S}^{1}(\Omega,\T))\lesssim h^{-2},
\end{equation}
see~\cite[Lem.~4.1]{Babuska1979}, which means that the corner refinements applied in graded meshes are sufficiently local as to preserve the order of the number of degrees of freedom occurring in uniform meshes. In particular, we can construct a sequence of (hierarchical graded) meshes $\{\mathcal{T}_N\}_{N\geq 0}$, enumerated by integer indices~$N\ge 0$, so that the corresponding sequence of finite element spaces $\VV_N:=\mathbb{S}^{1}(\Omega,\mathcal{T}_N)$ fulfills Asm.~\ref{ass:Galerkin density} as $N\to\infty$.

\subsection{Optimal rate under particular conditions}\label{sc:OCR cases}
    Throughout this section, we assume that Asm.~\ref{Ass: f} and Asm.~\ref{Ass:convergence} are satisfied for some $\rho>0$ and $\Delta t \in \Lambda_f(\rho)$. For each index $N\geq 0$, we select $\VV_N:=\mathbb{S}^{1}(\Omega,\mesh_N)$ as above, with a suitable mesh $\mesh_N$ of size $h(N)>0$, cf.~\eqref{eq:dim S1}. We construct a sequence of Galerkin approximations $\{\un\}_{N\geq 0}$, cf.~\eqref{eq:un}, and a limit $u\in \VV\cap \htto$ as in Prop.~\ref{prop: u critical Galerkin}. Since discrete solutions lie within finite-dimensional spaces $\VV_N$ for each sufficiently large $N$, it is meaningful for the subsequent discussion to suppose throughout that 
    \begin{equation}\label{eq:sol}
    u\not\in \bigcup_{N\geq 0}\VV_N.
    \end{equation}
    
    We now analyze three distinct contexts that provide a structured framework to demonstrate \eqref{eq:OCR} in the general case. These cases depend on the properties of the nonlinearity $f$ in~\eqref{Pb:main PDEa}, and on their respective impact on the energy functional $\E$ from~\eqref{eq:energy func}.

\subsubsection{Contraction regime}\label{ssc:OCR special case}
    We begin with the special case $\rho \leq \nicefrac{1}{\CP}$, where the contraction property \eqref{eq:stability bound} holds. Indeed, this situation is of particular interest as it allows for a possibly discontinuous derivative of the nonlinearity, which is not admissible in general.

    \begin{proposition}[Quasi-optimality in the special case~$\rho\le\nicefrac{1}{\CP}$]\label{prop:Cea special}
        Given an arbitrary closed linear subspace $\WW\subset \VV$. Then, for any initial guess $u^0\in\L^2(\Omega)$, the sequence $\{u^n\}_{n\geq 0}\subset \WW$ generated by the iteration~\eqref{eq:forms iteration scheme W} has a limit $u^\star\in \WW$ that solves~\eqref{Pb:weak form W}. Furthermore, there exists a constant $C(\Delta t)>0$ such that
        \begin{equation}\label{eq:Cea special}
            \|u-u^\star\|_{\Delta t} \leq C(\Delta t)\inf_{w\in \WW}\|u-w\|_{\Delta t},
        \end{equation}
        where $u\in \VV\cap \htto$ is the unique solution of \eqref{Pb:weak form}.
    \end{proposition}
    \begin{proof}
        Recalling Rmk.~\ref{rmk:special case} and letting $n\to\infty$ in the iteration~\eqref{eq:forms iteration scheme W}, we have $u^n\to u^\star$ strongly in~$\VV$, with
        \[
        \B_{\Delta t}(u^\star,w)=\ell_{\Delta t}(u^\star;w)\qquad\forall w\in\WW,
        \]
        which, in turn, yields~\eqref{Pb:weak form W}. Furthermore, we observe that the solution of~\eqref{Pb:main PDE} solves
        \[
        \B_{\Delta t}(u,v)=\ell_{\Delta t}(u;v)\qquad\forall v\in\VV.
        \]
        Hence, for $w\in\WW$ arbitrary, noticing that $w-u^\star\in\WW$, we decompose the error as
        \begin{align*}
            \|u-u^\star\|^2_{\Delta t}
            &=\B_{\Delta t}(u-u^\star,u-u^\star)\\
            &=\B_{\Delta t}(u-u^\star,u-w)+\B_{\Delta t}(u-u^\star,w-u^\star)\\
            &=\B_{\Delta t}(u-u^\star,u-w) + \ell_{\Delta t}(u;w-u)-\ell_{\Delta t}(u^\star;w-u)\\&\quad  + \ell_{\Delta t}(u;u-u^\star)-\ell_{\Delta t}(u^\star;u-u^\star).
        \end{align*}
        Then, by means of the Cauchy-Schwarz inequality, we obtain
        \begin{equation*}
            |\B_{\Delta t}(u-u^\star,u-w)| 
            \leq \|u-u^\star\|_{\Delta t}\|u-w\|_{\Delta t}.
        \end{equation*}
        Furthermore, using~\eqref{eq:l form},
        we have
        \begin{equation*}
            |\ell_{\Delta t}(u;w-u)-\ell_{\Delta t}(u^\star;w-u)| \leq \Delta t \into |g_{\Delta t}(\x,u)-g_{\Delta t}(\x,u^\star)|\,|w-u|\;\dx.
        \end{equation*}
        Moreover, consecutively employing \eqref{eq: g Lipschitz}, the Cauchy-Schwarz inequality, and \eqref{eq:Poincare lambda}, we deduce
        \begin{align*}
            |\ell_{\Delta t}(u;w-u)-\ell_{\Delta t}(u^\star;w-u)| 
            &\leq r(\Delta t) \|u-u^\star\|_{\Delta t}\|w-u\|_{\Delta t},
        \end{align*}
        with $0<r(\Delta t)<1$ from \eqref{eq:stability bound}. Similarly,
        \begin{equation}\label{eq:Cea contraction}
            |\ell_{\Delta t}(u;u-u^\star)-\ell_{\Delta t}(u^\star;u-u^\star)| \leq r(\Delta t) \|u-u^\star\|^2_{\Delta t}.
        \end{equation}
        Therefore, by the triangle inequality, it follows that
        \begin{equation*}
            \|u-u^\star\|^2_{\Delta t} \leq \left[1+r(\Delta t) \right]\|u-u^\star\|_{\Delta t}\|u-w\|_{\Delta t} + r(\Delta t) \|u-u^\star\|^2_{\Delta t},
        \end{equation*}
        which implies that
        \begin{equation}\label{eq:cea constant}
           \|u-u^\star\|_{\Delta t} \leq \frac{1+r(\Delta t)}{1-r(\Delta t)} \|u-w\|_{\Delta t}.
        \end{equation}
        Exploiting that $w \in \WW$ is arbitrary, we get \eqref{eq:Cea special}.
    \end{proof}
    
    Consequently, the quasi-optimality bound yields the desired optimal rate.

    \begin{corollary}[Optimal rate in the special case~$\rho\le\nicefrac{1}{\CP}$]\label{Cor: OCR Cea special}
        For a graded mesh family $\{\mesh_N\}_{N\ge 0}$, consider the corresponding sequence of finite element spaces $\VV_N = \mathbb{S}^{1}(\Omega,\mesh_N)$, cf.~\eqref{eq:fes}. Then, for each $N\ge 0$, the error between the exact solution $u\in\VV$ of~\eqref{Pb:main PDE} and its Galerkin approximation $\un\in\VV_N$ from~\eqref{eq:un} (which results as the limit of the iteration~\eqref{eq:forms iteration scheme W} on $\WW=\VV_N$) satisfies the bound
        \begin{equation*}
            \|u-u_N^\star\|_{\Delta t} \leq C\dim(\VV_N)^{-\nicefrac{1}{2}}\|u\|_{\htto},\qquad N\ge 0,
        \end{equation*}
        for a constant $C>0$ depending only on $\Omega$, $\Delta t$, and on the triangulation parameters $\boldsymbol{\beta}$ and $\kappa$.
    \end{corollary}
    \begin{proof}
        Using \eqref{eq:Cea special} for $\WW=\VV_N$, and employing \eqref{eq:Poincare L2}, we notice that
        \begin{equation*}
            \|u-\un\|_{\Delta t} \leq C(\Delta t)
            \|u-I_hu\|_{\Delta t}\leq C(\Delta t)(\CP+\Delta t)^{\nicefrac12}\|\grad(u-I_hu)\|_{\L^2(\Omega)}.
        \end{equation*}
        Now, the result follows from the interpolation bound \eqref{eq:graded estimate}.
    \end{proof}

\subsubsection{Gradient-dominating regime}\label{ssc:OCR grad}
    We now return to the case where $\rho > 0$ may be arbitrarily large. From the Poincar\'e inequality \eqref{eq:Poincare L2}, we know that the $\L^2$-norm decays at least as quickly as the gradient norm. Therefore, we will analyze the situation where the first variation \eqref{Pb:weak form} of the energy is primarily dominated by the latter. The following result establishes the desired rate \eqref{eq:OCR} in that case.
    \begin{proposition}[Optimal rate for dominating gradients]\label{prop:an faster}
        Suppose there exists $\varepsilon \in (0,1)$ such that for all sufficiently large $N\geq 0$ it holds
        \begin{equation}\label{eq:grad dominance}
            0 \leq  \an  \leq \varepsilon \bn^2,
        \end{equation}
        where $\un\in\VV_N$ solves~\eqref{eq:un}, and $u\in\VV$ is a solution of~\eqref{Pb:main PDE}.
        If $f_u$ is a Carathéodory function, cf. \eqref{eq:Caratheodory}, then 
        \begin{equation*}
            \bn \leq C\dim(\VV_N)^{-\nicefrac{1}{2}} \|u\|_{\htto},
        \end{equation*}
        for a constant $C>0$ depending only on $\Omega$, $\Delta t$, $f$, and on the triangulation parameters $\boldsymbol{\beta}$ and $\kappa$.
    \end{proposition}
    \begin{proof}
        Thanks to the triangle inequality it holds
        \begin{align}\label{eq: decomp gradients}
            (1-\varepsilon)\bn^2 &\leq \bn^2 - \an \nonumber \\ 
            & \leq |\dn|= |\dni|,
        \end{align}
        where we have used~\eqref{eq:Galerkin orthogonality} to replace $\un\in \VV_N$ by $I_hu$ in the duality product. By means of Lem.~\ref{lemma: f Lipschitz} and the Poincar\'e inequality \eqref{eq:Poincare L2}, we deduce the bound
        \begin{equation*}
            \ani \leq \CP(\rho+\mu_f) \bn \bni.
        \end{equation*}
        Thus, by the triangle inequality and the Cauchy-Schwarz inequality
        \begin{equation*}
            (1-\varepsilon)\bn \leq [(1+\CP(\rho+\mu_f)] \bni.
        \end{equation*}
        Dividing by $1-\varepsilon$, for $0<\varepsilon<1$, and applying the estimate \eqref{eq:graded estimate} completes the proof.
    \end{proof}
    
    At this point, it is of particular importance to recall that, in the special case analyzed in \S\ref{ssc:OCR special case}, we were able to argue for the optimal rate of convergence without further constraints because the whole sequence $\{\un\}_{N\geq 0}$ exhibited strong convergence in $\VV$. Similarly, the rate established by Prop.~\ref{prop:an faster} relies on a specific property held by the \textit{entire} sequence $\{\un\}_{N\geq 0}$ within $\VV$. However, according to Thm.~\ref{thm: convergence in closed sbspc}, we may only afford weak convergence in $\VV$ in general, and it is even possible that the gradient norm of the error might not converge at all. This is why the subsequent corollary becomes crucial for the remainder of our work, as it ensures that the gradient of the error does decay to zero \textit{for a subsequence}. In other words, Prop.~\ref{prop:an faster} implies that a subsequence of $\{\un\}_{N\geq 0}$ converges strongly to a solution~$u\in\VV$ of~\eqref{Pb:main PDE}, even in the general case to be studied later on in~\S\ref{sc:OCR general case}.
    
    \begin{corollary}[Strong convergence]\label{cor:strong conv}
        If $f_u$ is a Carathéodory function, cf. \eqref{eq:Caratheodory}, then there exists a subsequence of $\{\un\}_{N\geq 0}$ that converges strongly in $\VV$ to a solution $u\in\VV$ of~\eqref{Pb:main PDE}.
    \end{corollary}
    \begin{proof}
        Following the same line of arguments as in the proof of Thm.~\ref{thm: convergence in closed sbspc}, there is a subsequence $\{\un\}_{N\geq 0}$ that converges strongly in $\L^2(\Omega)$ and weakly in $\VV$ to a function $u\in\VV$, which is a solution of~\eqref{Pb:main PDE} thanks to Prop.~\ref{prop: u critical Galerkin}. Suppose now that there is no subsequence of $\{\un\}_{N\geq 0}$ that converges strongly to $u$ in $\VV$. Then, up to extracting a subsequence, it holds that 
        \begin{equation*}
            \lim\limits_{N\rightarrow \infty}\bn = c,
        \end{equation*}
        for some finite $c>0$, as the sequence remains bounded thanks to its weak convergence in $\VV$. In addition, owing to Lem.~\ref{lemma: f Lipschitz}, we have that
        \begin{equation*}
            0\leq \an \leq (\rho+\mu_f) \|u-\un\|_{\L^2(\Omega)}^2\rightarrow 0
        \end{equation*}
        as $N\rightarrow \infty$. As a result, for any $\varepsilon\in (0,1)$, we can find an integer $N_{\varepsilon,c}\in \N$ such that for all $N \geq N_{\varepsilon,c}$ we have
        \begin{equation*}
            0 \leq \an \leq \varepsilon\bn^2.
        \end{equation*}
        But since $u\in \htto$, Prop.~\ref{prop:an faster} then implies that
        $\bn \to 0$ as $N\rightarrow \infty$, which yields a contradiction.
    \end{proof}

\subsubsection{Asymptotic monotonicity}\label{ssc:OCR monotone}
    Whilst in the special case we could rely on the quasi-optimality bound \eqref{eq:Cea special} to demonstrate the optimal convergence rate, this approach cannot be pursued if $r(\Delta t) \ge 1$ in the bound \eqref{eq:Cea contraction}, since this would subvert the contraction behavior of the form $\ell_{\Delta t}$, and, in turn, the bound~\eqref{eq:cea constant} would no longer hold. Nonetheless, we can recover the result \textit{for a subsequence} when $f$ asymptotically exhibits a monotone behavior along the sequence $\{\un\}_{N\geq 0}$, in the sense that
    \begin{equation}\label{eq: f monotonic}
        \into (f(\x,u)-f(\x,\un))(u-\un) \dx < 0 \qquad \forall N\geq 0 \text{ large enough}.
    \end{equation}
    
    \begin{proposition}[Quasi-optimality under asymptotic monotonicity]\label{prop: Cea general}
        Assume that $f$ is asymptotically monotone in the sense of \eqref{eq: f monotonic}, and that $f_u$ is a Carathéodory function, cf. \eqref{eq:Caratheodory}. Then, there exists a subsequence of $\{\un\}_{N\geq 0}$ for which the quasi-optimality bound \eqref{eq:Cea special} holds.
    \end{proposition}
    \begin{proof}
        By Cor.~\ref{cor:strong conv}, we have strong convergence of $\{\un\}_{N\geq 0}$ in $\VV$, up to extracting a subsequence. Pick $N\geq 0$ so that \eqref{eq: f monotonic} is fulfilled, and let $w\in \VV_N$ be fixed. Following along the lines of the proof of Prop.~\ref{prop:Cea special} with $\WW=\VV_N$, we have
        \begin{align*}
            \|u-\un\|^2_{\Delta t}&\leq(1+r(\Delta t))\|u-\un\|_{\Delta t}\|u-w\|_{\Delta t}+ \ell_{\Delta t}(u;u-\un)-\ell_{\Delta t}(\un;u-\un),
        \end{align*}
        with $r(\Delta t)$ from~\eqref{eq:rt}, and $\ell_{\Delta t}$ from \eqref{eq:l form}. Now, because of the monotonicity \eqref{eq: f monotonic} and the Poincar\'e inequality \eqref{eq:Poincare lambda}, we note that
        \begin{align*}
            \ell_{\Delta t}(u;u-\un)-\ell_{\Delta t}(\un;u-\un) 
            &< \|u-\un\|_{\L^2(\Omega)}^2 \leq \theta(\Delta t) \|u-\un\|^2_{\Delta t},
        \end{align*}
        with $0<\theta(\Delta t)<1$, cf. \eqref{def:beta(lambda)}. Thus,
        \begin{equation*}
           \|u-\un\|_{\Delta t} \leq \frac{1+r(\Delta t)}{1-\theta(\Delta t)} \|u-w\|_{\Delta t}.
        \end{equation*}
        Exploiting that $w \in \VV_N$ is arbitrary, we deduce \eqref{eq:Cea special}.
    \end{proof}

    \begin{remark}\label{rmk:OCR ell<0}
        The optimal convergence rate is a direct consequence of the arguments given in \S\ref{ssc:OCR special case}. Indeed, in the monotone case, combining the quasi-optimality bound \eqref{eq:Cea special} and the interpolation estimate \eqref{eq:graded estimate} yields~\eqref{eq:OCR} for a subsequence of $\{\un\}_{N\geq 0}$. The constant in \eqref{eq:OCR} depends only on $\Omega$, $\Delta t$, and the triangulation parameters $\boldsymbol{\beta}$ and $\kappa$.
    \end{remark}

\subsection{The general case}\label{sc:OCR general case}
    We now proceed to derive \eqref{eq:OCR} for a subsequence in the general setting. 
    To this end, we introduce two sequences, namely $\{a_N\}_{N\geq 0}$, $\{b_N\}_{N\geq 0} \subset \R$, defined by
        \begin{equation}\label{eq:def an bn}
            a_N := \into (f(\x,u)-f(\x,\un))(u-\un) \dx \qquad \text{and} \qquad b_N := \bn^2.
        \end{equation}
        We notice from the Lipschitz continuity \eqref{eq: f Lipschitz} of $f$ and from the Poincar\'e inequality \eqref{eq:Poincare L2} that 
        $
            |a_N| \leq \CP(\rho+\mu_f)b_N.
        $
        Furthermore, under condition~\eqref{eq:sol}, we have $b_N > 0$ for all~$N$, 
        and hence
        \[
            0 \leq \left|\nicefrac{a_N}{b_N}\right| \leq \CP(\rho+\mu_f)\qquad\forall N\ge 0.
        \]
        Moreover, again by Cor.~\ref{cor:strong conv}, we have strong convergence of the sequence $\{\un\}_{N\geq 0}$ in~$\VV$, up to extracting a subsequence.
        Therefore, there is a (possibly further) subsequence of subspaces $\{\VV_{N'}\}_{N'}$ such that $\nicefrac{a_{N'}}{b_{N'}}\rightarrow \varsigma$ as $N'\rightarrow \infty$, for some limit $\limab \in \R$. Indeed, we may assume without loss of generality that the full sequence is convergent and that
        \begin{equation}\label{eq:def ell}
            \limab:= \lim\limits_{N\rightarrow \infty}\frac{a_N}{b_N}
        \end{equation}
        exists and is finite.
    
    \begin{proposition}[The case $\limab\le 0$]
    Suppose $f_u$ is a Carathéodory function, cf. \eqref{eq:Caratheodory}, and that $\limab\le 0$ in~\eqref{eq:def ell}. Then, up to extracting a subsequence, it holds
        \begin{equation}\label{eq:OCR general negative}
            \|u-u_N^\star\|_{\Delta t} \leq C\dim(\VV_N)^{-\nicefrac{1}{2}}\|u\|_{\htto},
        \end{equation}
        for a constant $C>0$ depending only on $\Omega$, $\Delta t$, $f$, and on the triangulation parameters $\boldsymbol{\beta}$ and $\kappa$.
    \end{proposition}

    \begin{proof}        
        In the case $\limab=0$, the gradient sequence $\{b_N\}_{N\geq 0}$ is dominant. More precisely, by definition of the limit \eqref{eq:def ell}, for any fixed $0<\varepsilon<1$ and $N\geq 0$ large enough, we have
        $
            \left|\nicefrac{a_N}{b_N}\right|<\varepsilon,
        $
        which corresponds exactly to~\eqref{eq:grad dominance}, and the rate \eqref{eq:OCR general negative} results immediately from Prop.~\ref{prop:an faster}.
        
        If $\limab<0$, then $a_N<0$ for every $N\geq 0$ large enough, since $b_N>0$. In particular, \eqref{eq: f monotonic} is satisfied, and we fall into the monotone case treated in \S\ref{ssc:OCR monotone}. As a result, the quasi-optimality bound \eqref{eq:Cea special} holds for a further subsequence, thanks to Prop.~\ref{prop: Cea general}, and \eqref{eq:OCR general negative} follows from Rmk.~\ref{rmk:OCR ell<0}.
    \end{proof}
      
    The main part of the error analysis in the general case consists in demonstrating by contradiction that the remaining case $\limab>0$ cannot happen. To this end, we will use the Aubin-Nitsche technique for semilinear PDEs proposed in \cite{Hardering2017}.
                
\subsubsection{Aubin-Nitsche duality}\label{ssc:Aubin-Nitsche}

For a fixed solution $u\in\VV$ of~\eqref{Pb:main PDE}, i.e., a critical point of the energy function $\E$, and a numerical approximation $\un\in\VV_N$, cf.~\eqref{eq:weak* convergence}, we consider the \emph{linear} dual problem
        \begin{equation}\label{eq:dual pb}
        \phi\in\VV:\qquad    -\Delta\phi-f_u(\cdot,u)\phi = u-\un\quad\text{in }\VV'.
        \end{equation}
In the special case $\rho\leq\nicefrac{1}{\CP}$, with $\CP$ from~\eqref{eq:Poincare L2}, the existence of a (unique) solution $\phi\in \VV$ to \eqref{eq:dual pb} is ensured by the Lax-Milgram theorem. Invertibility of the duality operator given by
\begin{equation}\label{eq:PP}
\PP:=-\Delta -f_u(\cdot,u)\,\Id:\,\VV\to\VV'\,,
\end{equation}
however, may fail in general. In this situation, applying~\eqref{eq:f' bounds}, we observe that the operator $\PP+\rho\,\Id$ is uniformly elliptic, and therefore, writing
$
    \PP=[\PP+\rho\,\Id]-\rho\,\Id,
$
we deduce that $\PP$
satisfies Fredholm's alternative, i.e., either $\KK:=\ker(\PP)=\{0\}$, in which case~\eqref{eq:dual pb} is uniquely solvable, or $\KK$ is finite-dimensional, and thus closed; see, e.g., \cite[Thm.~9.21]{RenardyRogers}. 
To shed some light on the latter case, let us consider the graph space $\DD:=\{v\in\VV:\,\Delta v\in\spcL^2(\Omega)\}$, equipped with the norm $\|v\|_{\DD}:=\|v\|_{\spcL^2(\Omega)}+\|\Delta v\|_{\spcL^2(\Omega)}$, which is a Banach subspace of $\VV$ by the closedness of the Laplacian. By the boundedness of $f_u$, cf.~\eqref{eq:f' bounds}, it holds that 
        \[
        \|\PP(v)\|_{\L^2(\Omega)}
        =\|-\Delta v+f_u(\cdot,u)v\|_{\L^2(\Omega)}
        \le \max(1,\rho+\mu_f)\|v\|_{\DD}\qquad\forall v\in\DD,
        \]
i.e., $\PP|_{\DD}:\,\DD\to\L^2(\Omega)$ is a bounded linear operator. Furthermore, since $\KK$ is closed, we can apply the $\L^2(\Omega)$-orthogonal decomposition
\begin{equation}\label{eq: L2 orthogonal decomposition}
    \L^2(\Omega)=\KK\oplus \KK^\perp,
\end{equation}
which permits to express the range as $\PP(\DD)=\KK^\perp$.
In summary, we conclude that the linear mapping
\begin{equation}\label{eq:PP1}
\PP|_{\DD^\perp_u}:\,\DD^\perp_u\to\KK^\perp
\end{equation}
is bijective on the closed linear subspace $\DD^\perp_u := \DD \cap \KK^\perp$. Accordingly, in the sequel, we will work with the assumption that
        \begin{equation}\label{eq:Fredholm}
        u-\un\in\KK^\perp,
        \end{equation}
which implies the existence of a unique solution $\phi\in\DD^\perp_u$ of the dual problem~\eqref{eq:dual pb}; clearly, if $\KK=\{0\}$, then $\KK^\perp=\L^2(\Omega)$, and the above condition~\eqref{eq:Fredholm} is redundant.

\begin{lemma}[Dual problem]\label{lemma:dual pb}
        Let $u\in \VV\cap \htto$ be a critical point of $\E$, and denote by $\{\un\}_{N\geq 0}$ the subsequence in \eqref{eq:weak* convergence} converging weakly in $\VV$ and strongly in $\L^2(\Omega)$ to $u$. Suppose that condition \eqref{eq:Fredholm} is satisfied for an index $N\geq 0$ associated to the subsequence. In addition, if
        Asm.~\ref{Ass: f}, Asm.~\ref{Ass:convergence} and Asm.~\ref{ass:Galerkin density} hold, then the dual problem~\eqref{eq:dual pb} possesses a solution $\phi\in\htto$ that satisfies the stability bound
        \begin{equation}\label{Elliptic regularity}
            \|\phi\|_{\htto} \leq \ER \|u-\un\|_{\L^2(\Omega)},
        \end{equation}
        for a constant $\ER>0$ independent of $\dim(\VV_N)$.
    \end{lemma}
    
    \begin{proof}
    From the analysis above, we recall that the linear operator $\PP$ from~\eqref{eq:PP} is bounded and invertible for the spaces $\DD^\perp_u\to\KK^\perp$, cf.~\eqref{eq:PP1}. Therefore,
        by the bounded inverse theorem, we deduce that its inverse $\PP^{-1}$ is bounded as well, i.e., there is a constant $C_{\ref{eq:ellstab}}>0$ such that
        \begin{equation}\label{eq:ellstab}
        \|\PP^{-1}(w)\|_{\DD}\le C_{\ref{eq:ellstab}}\|w\|_{\L^2(\Omega)}\qquad\forall w\in\KK^\perp.
        \end{equation}
        Especially, using~\eqref{eq:Fredholm}, the function $\phi:=\PP^{-1}(u-\un)\in\DD^\perp_u\subset\VV$ solves~\eqref{eq:dual pb}, and it holds the bound
        \begin{equation}\label{eq:ellbd}
        \|\phi\|_{\L^2(\Omega)}\le C_{\ref{eq:ellstab}}\|u-\un\|_{\L^2(\Omega)}.
        \end{equation}
        Moreover,
        by elliptic regularity in weighted Sobolev spaces, see~\cite{Babuska1979}, we have the continuous inclusion $(-\Delta)^{-1}\left(\spcL_{\bm{\mathsf{\beta}}}(\Omega)\right) \subset \htto$; therefore, noticing that $\L^2(\Omega)\subset \L_{\bm\beta}(\Omega)$, there is a constant $C_\Delta>0$ such that
        \begin{equation*}
        \|(-\Delta)^{-1}w\|_{\htto}\le C_{\Delta}\|w\|_{\L^2(\Omega)}\qquad\forall w\in\L^2(\Omega).
        \end{equation*}
        Hence, 
        \[
        \|\phi\|_{\htto}
        =\|(-\Delta)^{-1}[f_u(\cdot,u)\phi+u-\un]\|_{\htto}
        \le C_\Delta \|f_u(\cdot,u)\phi+u-\un\|_{\L^2(\Omega)}.
        \]
        Exploiting the boundedness of~$f_u(\cdot, u)$, and invoking~\eqref{eq:ellbd}, gives~\eqref{Elliptic regularity}. 
    \end{proof}

\begin{remark}\label{rmk:aux pb existence}
Letting $\Tilde{f}(\cdot,\varphi):= f_u(\cdot,u)\varphi + u-\un$, we see that the energy $\E^d$ associated with the dual problem~\eqref{eq:dual pb} takes a similar structure as the primal energy~$\E$ from~\eqref{eq:energy func}, when $f$ is replaced by $\Tilde{f}$, viz.
\begin{equation*}
\E^d(\varphi)
=\frac12\into|\grad\varphi|^2\,\dx-\frac12\into f_u(\x,u)\varphi^2\,\dx-\into(u-\un)\varphi\,\dx\,;
\end{equation*}
incidentally, $\Tilde{f}$ satisfies Asm.~\ref{Ass: f} with the same bounds \eqref{eq:f' bounds}. Therefore, in light of Prop.~\ref{prop: u critical Galerkin}, the existence of a dual solution can be ensured whenever $\E^d$ is weakly coercive. To this end, for a solution $u\in\VV$ of~\eqref{Pb:main PDE} and any $\varphi\in\VV$, recalling~\eqref{eq:energy func}, we observe that
\begin{align*}
\E^d(\varphi)
&=\E(u+\varphi)-\E(u)-\into(u-\un)\varphi\,\dx\\
&\quad+\into\left(F(\x,u+\varphi)-F(\x,u)-f(\x,u)\varphi-\frac12f_u(\x,u)\varphi^2\right)\dx
.
\end{align*}
Therefore, since $\E$ is weakly coercive by Asm.~\ref{Ass:convergence}~(ii), we obtain $\E^d(\varphi)\to+\infty$ as $\|\varphi\|_{\VV}\to\infty$ provided that the last integral grows faster than $\|\varphi\|_{\L^{p}(\Omega)}$, for some fixed $p\in(1,\infty)$, thereby absorbing the integral involving the error $u-\un$ by application of H\"older's inequality and the Sobolev embedding $\spcH^1(\Omega)\hookrightarrow\L^{\nicefrac{p}{(p-1)}}(\Omega)$ in two-space dimensions, viz. 
\[
\left|\into(u-\un)\varphi\,\dx\right|
\le \|u-\un\|_{\L^{\nicefrac{p}{(p-1)}}(\Omega)}\|\varphi\|_{\L^p(\Omega)}
\lesssim \|u-\un\|_{\VV}\|\varphi\|_{\L^p(\Omega)}.
\]
In order to formulate a sufficient condition, for \emph{any} $\delta>0$, suppose that there is $R_\delta>0$ such that
\[
\into\left(F(\x,u+\varphi)-F(\x,u)-f(\x,u)\varphi-\frac12f_u(\x,u)\varphi^2\right)\dx
\ge \delta\|\varphi\|_{\L^{p}(\Omega)},
\]
whenever $\|\varphi\|_{\VV}\ge R_\delta$. Using a Taylor expansion, under appropriate differentiability, this can be achieved if
\[
\into f_{uu}(\x,\varphi)\varphi^3 \, \dx\ge \delta\|\varphi\|_{\L^{p}(\Omega)},
\]
for any $\delta>0$ and $\|\varphi\|_{\VV}\ge R_\delta$.
\end{remark}

        \subsubsection{Error analysis for $\limab>0$}
        
        The subsequence analysis requires the following technical result.

    \begin{lemma}\label{Lemma: int f' conv}
        For $\x\in \Omega$, consider the path $\Gamma_t(\x):= tu(\x)+(1-t)\un(\x)$ with $t\in [0,1]$, where $u\in\VV$ and $\un\in\VV_{N}$ are solutions of~\eqref{Pb:main PDE} and~\eqref{eq:un}, respectively, with $\un\rightharpoonup u$ in $\VV$ and $\un\to u$ in $\L^2(\Omega)$ as $N\to\infty$ (possibly up to a subsequence), cf.~Thm.~\ref{thm: convergence in closed sbspc}. If $f_u$ is a Carathéodory function, cf. \eqref{eq:Caratheodory}, then, up to extracting a (further) subsequence,
        \begin{equation}\label{eq: int f' conv}
            \lim_{N\rightarrow \infty}\into |f_u(\x,\Gamma_t)-f_u(\x,u)|^2\,\dx =0,
        \end{equation}
        uniformly with respect to $t\in [0,1]$.
    \end{lemma}
    \begin{proof}
        Fix an arbitrary $t\in [0,1]$. Since $\{\un\}_{N\geq 0}$ converges strongly in $\L^2(\Omega)$ for a subsequence, we can extract a further subsequence such that $u^\star_{N_j}(\x)\rightarrow u(\x)$ for almost every $\x\in \Omega$ as $j\rightarrow\infty$, and in turn $\Gamma_t \rightarrow u$ almost everywhere. 
        Hence, defining
        \begin{equation*}
            g_{N_j}(\x) := |f_u(\x,\Gamma_t(\x)) - f_u(\x,u(\x))|^2,
        \end{equation*}
        the subsequence $\{g_{N_j}\}_{j\geq 0}$ 
        converges to $0$ almost everywhere, thanks to the continuity of $f_u$. Furthermore, 
        by the bound on $f_u$ from \eqref{eq:f' bounds} and the triangle inequality we deduce
        \begin{equation*}
            0\leq \into g_{N_j}(\x) \, \dx \leq \into4(\rho+\mu_f)^2 \,\dx= 4|\Omega|(\rho+\mu_f)^2<\infty,
        \end{equation*}
        where the measure $|\Omega|$ is finite since $\Omega$ is bounded; this shows that each $g_{N_j}$ is integrable and dominated by a constant integrable function (independent of $t$). Applying the dominated convergence theorem yields
        \begin{equation*}
            \lim_{j\rightarrow \infty}\into g_{N_j}(\x)\; \dx =0,
        \end{equation*}
        which demonstrates \eqref{eq: int f' conv}.
    \end{proof}
    
        We are now ready to attend to the case $\limab>0$ in~\eqref{eq:def ell}.
        A straight-forward computation reveals that the second variation of the energy functional~$\E$ from~\eqref{eq:energy func} at $u\in\VV$ is given by
        \begin{equation*}
            \E''(u)(\varphi,v) = \into \grad \varphi\cdot \grad v \;\dx- \into f_u(\x,u)\varphi v \;\dx \qquad \forall \varphi,v \in \VV,
        \end{equation*}
        and since $f_u$ is bounded,
        the second variation is well-defined on $\VV$. In particular, the weak formulation of \eqref{eq:dual pb} becomes
        \begin{equation}\label{eq:2nd variation}
            \E''(u)(\phi,v) = \into (u-\un)v\,\dx \qquad \forall v\in \VV.
        \end{equation}
        Now, let us denote by $I_h\phi$ the nodal interpolation of $\phi$ from~\eqref{eq:graded estimate} over $\VV_N$. Picking $v=u-\un$ in \eqref{eq:2nd variation}, we have
        \begin{equation}\label{eq:2nd var decomp}
            \|u-\un\|^2_{\L^2(\Omega)}=
            \E''(u)(\phi,u-\un)=\E''(u)(\phi-I_h\phi,u-\un)+\E''(u)(I_h\phi,u-\un).
        \end{equation}
        Applying the triangle inequality to the first term of the right-hand side then yields
        \begin{equation}\label{eq:1st term}
             |\E''(u)(\phi-I_h\phi,u-\un)| \leq \into |\grad(\phi-I_h\phi) ||\grad(u-\un)|  \dx + \into |f_u(\x,u)|\, |\phi-I_h\phi|\,|u-\un|  \dx.
        \end{equation}
        By combining H\"older's inequality, \eqref{eq:f' bounds} and \eqref{eq:Poincare L2}, the second integral of \eqref{eq:1st term} is bounded by
        \begin{equation}\label{eq:1st term 2nd term}
            \into |f_u(\x,u)|\, |\phi-I_h\phi|\,|u-\un| \; \dx \leq (\rho+\mu_f)\CP^{\nicefrac{1}{2}}\|\grad(\phi-I_h\phi)\|_{\L^2(\Omega)}\|u-\un\|_{\L^2(\Omega)}.
        \end{equation}
        Furthermore, we recall from \eqref{eq:def ell} that $\nicefrac{b_N}{a_N}\rightarrow \limab^{-1} \in (0,\infty)$ as $N\rightarrow \infty$. Thus, for all $\varepsilon>0$ and all sufficiently large $N\geq 0$, it holds $\left|\nicefrac{b_N}{a_N}-\limab^{-1}\right|<\varepsilon$, which implies that
        \begin{equation*}
            0< b_N < \left(\limab^{-1}+\varepsilon\right)a_N;
        \end{equation*}
        we stress here that the right-hand side is indeed a strictly positive upper bound, since $\limab>0$ by assumption, and in turn, $a_N>0$ for all $N$ large enough. Employing~\eqref{eq:def an bn}, and applying the bound \eqref{eq: f Lipschitz}, the above estimate translates into
        \begin{equation}\label{eq:grad equiv L2}
            \|\grad(u-\un)\|_{\L^2(\Omega)}^2 < \left(\limab^{-1}+\varepsilon\right)(\rho+\mu_f)\|u-\un\|^2_{\L^2(\Omega)},
        \end{equation}
        and so by the Cauchy-Schwarz inequality
        \begin{equation}\label{eq:1st term 1st term}
            \into |\grad(\phi-I_h\phi)||\grad(u-\un)|  \dx \leq \left(\limab^{-1}+\varepsilon\right)^{\nicefrac{1}{2}}(\rho+\mu_f)^{\nicefrac{1}{2}}\|\grad(\phi-I_h\phi)\|_{\L^2(\Omega)}\|u-\un\|_{\L^2(\Omega)}.
        \end{equation}
        As a result, plugging \eqref{eq:1st term 2nd term} and \eqref{eq:1st term 1st term} into \eqref{eq:1st term}, gives
        \begin{equation}\label{eq:bound 1st term}
            |\E''(u)(\phi-I_h\phi,u-\un)| \leq C_{\ref{eq:bound 1st term}}\|\grad(\phi-I_h\phi)\|_{\L^2(\Omega)}\|u-\un\|_{\L^2(\Omega)},
        \end{equation}
        where $C_{\ref{eq:bound 1st term}}:=\left(\limab^{-1}+\varepsilon\right)^{\nicefrac{1}{2}}(\rho+\mu_f)^{\nicefrac{1}{2}}+(\rho+\mu_f)\CP^{\nicefrac{1}{2}}>0$ is independent of $N$. Finally, we conclude from the estimates \eqref{eq:graded estimate} and \eqref{Elliptic regularity} that
        \begin{equation}\tag{\ref{eq:bound 1st term}'}\label{eq:tilde bound 1st term}
            |\E''(u)(\phi-I_h\phi,u-\un)|
            \leq C_{\ref{eq:bound 1st term}}\GM \ER  \dim(\VV_N)^{-\nicefrac{1}{2}} \|u-\un\|^2_{\L^2(\Omega)}.
        \end{equation}
        
        We now treat the second term of \eqref{eq:2nd var decomp}. To this end, recall that $u$ is a solution of \eqref{Pb:weak form}, and that $\un\in\VV_N$ satisfies~\eqref{eq:un}. Hence, by the main theorem of calculus, 
        \begin{equation*}
            \int_0^1 \E''(\Gamma_t)(v,u-\un) \; \dt = \int_0^1 \frac{\d}{\dt}\la \E'(\Gamma_t),v\ra \; \dt = \la \E'(u)-\E'(\un),v\ra =0 \qquad \forall v \in \VV_N,
        \end{equation*}
        where $\Gamma_t:=tu+(1-t)\un$, cf.~Lem.~\ref{Lemma: int f' conv}. As a result, choosing $v=I_h\phi\in\VV_N$, we obtain
        \begin{align*}
            \E''(u)(I_h\phi,u-\un) &= \int_0^1 \left(\E''(u)(I_h\phi,u-\un)- \E''(\Gamma_t)(I_h\phi,u-\un)\right)  \dt \\
            &= \int_0^1\left(\into (f_u(\x,\Gamma_t)-f_u(\x,u))I_h\phi(u-\un) \dx \right) \dt,
        \end{align*}
        whence H\"older's inequality implies
        \begin{equation*}
            |\E''(u)(I_h\phi,u-\un)|\leq \|I_h\phi\|_{\L^{\infty}(\Omega)}\|u-\un\|_{\L^2(\Omega)}\int_0^1 \left(\into |f_u(\x,\Gamma_t)-f_u(\x,u)|^2 \; \dx \right)^{\nicefrac{1}{2}} \dt.
        \end{equation*}
        Recalling the bound \eqref{eq:Istab1}, we obtain
        \begin{equation*}
            \|I_h\phi\|_{\mathrm{C}^0(\Omega)} 
            \leq C\|\phi\|_{\htto} \leq C\ER \|u-\un\|_{\L^2(\Omega)}.
        \end{equation*}
        Thus,
        \begin{equation*}
            |\E''(u)(I_h\phi,u-\un)|\leq C\ER \|u-\un\|^2_{\L^2(\Omega)}\sup_{t\in[0,1]} \left(\into |f_u(\x,\Gamma_t)-f_u(\x,u)|^2  \dx \right)^{\nicefrac{1}{2}}.
        \end{equation*}
        Next, let us fix some 
        \begin{equation}\label{eq:choice}
        0<\varepsilon'<\nicefrac{1}{C\ER},
        \end{equation}
        and, up to extracting a further subsequence according to Lem.~\ref{Lemma: int f' conv}, choose $N\geq 0$ large enough such that
        \begin{equation}\label{eq:contr}
            \sup_{t\in[0,1]} \left(\into |f_u(\x,\Gamma_t)-f_u(\x,u)|^2 \; \dx \right)^{\nicefrac{1}{2}} < \varepsilon';
        \end{equation}
        cf. Lem.~\ref{Lemma: int f' conv}. Then,
        \begin{equation}\label{eq:bound 2nd term}
            |\E''(u)(I_h\phi,u-\un)|\leq \varepsilon'C\ER \|u-\un\|^2_{\L^2(\Omega)}.
        \end{equation}
        Altogether, we conclude from \eqref{eq:tilde bound 1st term} and \eqref{eq:bound 2nd term} that the decomposition \eqref{eq:2nd var decomp} is bounded by
        \begin{equation*}
            \|u-\un\|^2_{\L^2(\Omega)} \leq C_{\ref{eq:bound 1st term}}\GM \ER \dim(\VV_N)^{-\nicefrac{1}{2}} \|u-\un\|^2_{\L^2(\Omega)} + \varepsilon'C\ER \|u-\un\|^2_{\L^2(\Omega)}.
        \end{equation*}
        In particular, because $0<\varepsilon'C\ER<1$ we obtain
        \begin{equation*}
            \|u-\un\|^2_{\L^2(\Omega)} \leq \frac{C_{\ref{eq:bound 1st term}}\GM \ER}{1-\varepsilon'C\ER } \dim(\VV_N)^{-\nicefrac{1}{2}} \|u-\un\|^2_{\L^2(\Omega)}.
        \end{equation*}
        For $N$ sufficiently large so that
        \begin{equation*}
            \dim(\VV_N) > \left(\frac{C_{\ref{eq:bound 1st term}}\GM \ER}{1-\varepsilon'C\ER}\right)^2,
        \end{equation*}
        we obtain a contradiction, and we conclude that $\limab>0$ cannot occur in~\eqref{eq:def ell}.

    \begin{remark}
        The assumptions in \cite{Hardering2017} on the dual problem \eqref{eq:dual pb} are stronger due to the greater generality of their setting, which considers a broad class of functionals $\E:\H^1(\Omega)\rightarrow \R$ with $\Omega \subset \R^d$. In contrast, our analysis focuses specifically on semilinear diffusion-reaction problems, allowing us to work under weaker conditions. Specifically, we neither require strong coercivity for the energy~$\E$, nor higher regularity for $f$, nor any additional bounds on the third variation of $\E$.
    \end{remark}

\subsubsection{Optimal convergence results}

    We now summarize our findings in the following theorem.

    \begin{theorem}[Optimal convergence--general case]\label{thm:OCR Galerkin}
        Let $\Omega\subset \R^2$ be an open, bounded, polygonal domain. Suppose that Asm.~\ref{Ass: f} and Asm.~\ref{Ass:convergence} hold true for some $\rho>0$ and $\Delta t \in \Lambda_f(\rho)$. Let $\{\mesh_N\}_{N\ge 0}$ be a graded mesh family as in Def.~\ref{def:triangulation type}, and assume that the associated sequence of finite element spaces $\VV_N = \mathbb{S}^{1}(\Omega,\mesh_N)$, cf.~\eqref{eq:fes}, fulfills Asm.~\ref{ass:Galerkin density} as $N\rightarrow \infty$. Let $\{\un\}_{N\geq 0}$ denote the subsequence from \eqref{eq:weak* convergence} that converges weakly in $\VV$ and strongly in $\L^2(\Omega)$ to $u\in\VV\cap\htto$, a critical point of the energy functional $\E$ defined in \eqref{eq:energy func}, and hence a weak solution of \eqref{Pb:main PDE}. Furthermore, suppose 
        that $f_u$ is a Carathéodory function in the sense of~\eqref{eq:Caratheodory}, and that condition~\eqref{eq:Fredholm} holds for the dual problem \eqref{eq:dual pb} (being redundant if $\KK=\{0\}$). Then, up to extracting a further subsequence, we have
        \begin{equation}\label{eq:OCR general}
            \|u-u_N^\star\|_{\Delta t} \leq C\dim(\VV_N)^{-\nicefrac{1}{2}}\|u\|_{\htto},
        \end{equation}
        for a constant $C>0$ depending only on $\Omega$, $\Delta t$, $f$, and on the triangulation parameters $\boldsymbol{\beta}$ and $\kappa$.
    \end{theorem}

Thanks to the Lipschitz continuity \eqref{eq: E' Lipschitz} of $\E'$, the residuals exhibit the same optimal rate of decay as the underlying sequence.

    \begin{corollary}[Optimal rate for residuals]\label{cor:residual OCR}
        Suppose that the hypotheses of Theorem~\ref{thm:OCR Galerkin} hold true. Then, the sequence of corresponding residuals $\{\E'(\un)\}_{N\geq 0}$ satisfies the error bound
        \begin{equation*}
            \|\E'(\un)\|_{\VV'} \leq C \dim(\VV_N)^{-\nicefrac{1}{2}}\|u\|_{\htto},
        \end{equation*}
        for a constant $C>0$ depending only on $\Omega$, $\Delta t$, $f$, and on the triangulation parameters $\boldsymbol{\beta}$ and $\kappa$.
    \end{corollary}
    \begin{proof}
        Using the fact that $u$ is a solution of \eqref{Pb:weak form}, and recalling~\eqref{eq:lambda norm}, implies
        \begin{equation*}
        \|\E'(\un)\|_{\VV'} 
        = \sup_{v\in \VV\backslash\{0\}}\frac{|\la \E'(\un)-\E'(u),v\ra|}{\|v\|_{\VV}} 
        \le\sqrt{\Delta t} \sup_{v\in \VV\backslash\{0\}}\frac{|\la \E'(\un)-\E'(u),v\ra|}{\|v\|_{\Delta t}}.
        \end{equation*}
        Thus, upon using Lem.~\ref{lemma: E' Lipschitz}, we arrive at
        \begin{equation*}
        \|\E'(\un)\|_{\VV'}\leq L_{\E'}(\Delta t)\sqrt{\Delta t}\|u-\un\|_{\Delta t}.
    \end{equation*}
    The claim now follows from Thm.~\ref{thm:OCR Galerkin}.
    \end{proof}

\subsubsection{Degenerate solutions}
    Recalling the decomposition \eqref{eq: L2 orthogonal decomposition}, we can generalize the optimal rate result from Thm.~\ref{thm:OCR Galerkin} to (not fully) degenerate solutions, i.e., to the case when the right-hand side of the dual problem~\eqref{eq:dual pb} is decomposed accordingly by
    \begin{equation}\label{eq:errdecomp}
    u-\un=\enpar+\enperp,\qquad\text{with } \enpar\in\KK,\, \enperp\in\KK^\perp,
    \end{equation}
    and $\enpar\neq 0$, the magnitude of which can be controlled uniformly.

    \begin{theorem}[Optimal convergence for degenerate solutions]\label{thm:OCR Galerkin deg}
        Under the assumptions of the previous Thm.~\ref{thm:OCR Galerkin}, suppose that the kernel component never asymptotically dominates the total error, i.e., there exists a constant $0\leq \delta <1$ (independent of $N$) such that, for the decomposition~\eqref{eq:errdecomp}, we have
        \begin{equation}\label{eq:non-degeneracy}
            \|\enpar\|_{\L^2(\Omega)} \leq \delta \|u-\un\|_{\L^2(\Omega)},
        \end{equation}
        for all sufficiently large~$N$.
        Then, up to extracting a subsequence, it holds
        \begin{equation*}
            \|u-u_N^\star\|_{\Delta t} \leq C\dim(\VV_N)^{-\nicefrac{1}{2}}\|u\|_{\htto},
        \end{equation*}
        for a constant $C>0$ depending only on $\Omega$, $\Delta t$, $f$, and on the triangulation parameters $\boldsymbol{\beta}$ and $\kappa$.
    \end{theorem}
    \begin{proof}
        Consider the dual problem
        \begin{equation*}
            -\Delta\phi-f_u(\cdot,u)\phi = \enperp\qquad \text{in }\VV',
        \end{equation*}
        which admits a solution $\phi\in \VV \cap \htto$ by Lem.~\ref{lemma:dual pb}, since condition~\eqref{eq:Fredholm} is satisfied with $\enperp\in \KK^\perp$. Similarly as in~\eqref{eq:2nd variation}, we have the identity
        \begin{equation}\label{eq:2nd variation deg}
            \E''(u)(\phi,v) = \la v,\enperp\ra_{\L^2(\Omega)} \qquad \forall v\in \VV.
        \end{equation}
        From the $\L^2(\Omega)$-orthogonal decomposition \eqref{eq: L2 orthogonal decomposition}, letting $v=u-\un$ in~\eqref{eq:2nd variation deg}, we deduce
        \begin{equation*}
            \|u-\un\|_{\L^2(\Omega)}^2  = \|\enpar\|_{\L^2(\Omega)}^2 + \la \enperp,u-\un\ra_{\L^2(\Omega)} = \|\enpar\|_{\L^2(\Omega)}^2 + \E''(u)(\phi,u-\un).
        \end{equation*} 
        Applying analogous arguments as in the proof of Thm.~\ref{thm:OCR Galerkin}, with the choice $0<\varepsilon'<\nicefrac{1-\delta^2}{C\ER}$, cf.~\eqref{eq:choice}, where $0\leq \delta<1$ is the constant from \eqref{eq:non-degeneracy}, we obtain a contradiction as soon as
        \begin{equation*}
            \dim(\VV_N) > \left(\frac{C_{\ref{eq:bound 1st term}}\GM \ER}{1-\varepsilon'C\ER -\delta^2}\right)^2,
        \end{equation*}
        cf.~\eqref{eq:contr}, i.e., for any $N\geq 0$ sufficiently large.
    \end{proof}

    \begin{remark}
        From an operational point of view, condition \eqref{eq:non-degeneracy} permits a ``$\delta$-percentage'' of the error to lie in the kernel $\KK$ while excluding the practically unrealistic case in which the error is fully degenerate and the (linearized) residual (and therefore any estimator built from it) becomes asymptotically blind to the orthogonal part $\|\enperp\|_{\L^2(\Omega)}$ of the error. In this sense, the hypothesis is both natural and moderate: On the one hand, it rules out complete loss of information in the refinement process, and, on the other hand, it is necessary for any estimate coming from the dual problem to yield control of the full error.
        
        At the same time, in the degenerate case, the parameter $\delta$ cannot be taken arbitrarily small in general. Indeed, denote by $\overline{\mu}$ the smallest eigenvalue of the operator $\PP$ from~\eqref{eq:PP}, whose corresponding eigenfunction lies in $\KK^\perp$. If $\overline{\mu}>\rho+\mu_f$, then from \eqref{eq:grad equiv L2}, we obtain the quantitative lower bound
        \begin{equation*}
            \delta^2 > 1- \frac{\left(\limab^{-1}+\varepsilon\right)(\rho+\mu_f)}{\overline{\mu}-(\rho+\mu_f)}.
        \end{equation*}
    \end{remark}

\section {Convergence of iterative Galerkin solutions}\label{sc:approx}

    In case that we do not have the exact Galerkin solutions $\{\un\}_{N\geq 0}$ from~\eqref{eq:un} at hand, which will be the typical situation in practice, it is still possible to get an arbitrarily close approximation to a weak solution 
    of~\eqref{Pb:main PDE} through a (finite) iterative process, and to achieve the optimal convergence rate from~\eqref{eq:OCR general}. 
    
    Specifically, given an initial guess $u^0_N\in \L^2(\Omega)$, we perform the iterative scheme \eqref{eq:forms iteration scheme W} on the finite element space $\WW=\VV_N$, generating the corresponding sequence $\{u^n_N\}_{n\geq 0}\subset \VV_N$. Since $\VV_N$ is finite-dimensional, the sequence converges strongly (up to a subsequence) to an exact Galerkin solution $\un\in \VV_N$ of \eqref{eq:un}, see Thm.~\ref{thm: convergence in closed sbspc}. Furthermore, for every fixed $\Delta t>0$ satisfying Asm.~\ref{Ass: f}, thanks to~\eqref{eq:Galerkin orthogonality} and Lem.~\ref{lemma: E' Lipschitz}, we have
    \begin{equation*}
        \|\E'(u^n_N)\|_{\VV'_N} := \sup_{v\in \VV_{N}} \frac{\la\E'(u^n_N)-\E'(\un),v\ra}{\|v\|_{\VV}} \leq L_{\E'}(\Delta t)\|u^n_N-\un\|_{\Delta t} \rightarrow 0,
    \end{equation*}
    as $n\rightarrow \infty$. In particular, for any choice of strictly positive real numbers $\{\varepsilon_N\}_{N\geq 0}$ with $\varepsilon_N \rightarrow 0$ as $N\rightarrow\infty$, we have the residual bound
    \begin{equation}\label{eq: finite residual epsilon}
         \|\E'(\unn)\|_{\VV'_N}
         \leq \varepsilon_N,\qquad N\ge 0,
    \end{equation}
    provided that the final iteration index on the current space $\VV_N$, signified by $n^\star = n^\star(N)$, is chosen sufficiently large; we will use the notation
    \begin{equation}\label{eq: sequence approximations}
        \u{}:=\unn, \qquad N\geq 0,
    \end{equation}
    to denote the corresponding approximation of~$\un\in\VV_N$.
    
    Similarly as before, we construct approximations $\{\u{}\}_{N\geq 0}$ on a sequence of hierarchical spaces $\{\VV_N\}_{N\ge 0}$, cf.~Asm.~\ref{ass:Galerkin density}, by restarting the iteration~\eqref{eq:forms iteration scheme W} on each enriched space $\VV_{N+1}$ with the canonically embedded initial guess 
    \begin{equation}\label{eq:u0b}
    u_{N+1}^0:=\u{}\in \VV_N\hookrightarrow\VV_{N+1},
    \end{equation}
    cf.~\eqref{eq:u0}. Under Asm.~\ref{Ass:convergence}, by applying the same line of arguments as in \S\ref{sc: convergence}, we may extract a subsequence, again denoted by $\{\u{}\}_{N\geq 0}$, which converges weakly in $\VV$ and strongly in $\L^2(\Omega)$ to some $\uu\in \VV$, and for which
    \begin{equation}\label{eq:weak* convergence approx}
        \lim_{N\rightarrow \infty}\la \E'(\u{}),v \ra = \la \E'(\uu),v\ra \qquad \forall v \in \VV.
    \end{equation}
    It remains to show that the above limit $\uu\in\VV$ is a weak solution of~\eqref{Pb:main PDE}.
    
    \begin{proposition}\label{prop: approx critical pt}
        Suppose Asm.~\ref{Ass: f}, Asm.~\ref{Ass:convergence} and Asm.~\ref{ass:Galerkin density} are satisfied for some $\rho>0$ and $\Delta t\in \Lambda_f(\rho)$. Moreover, consider a sequence of strictly positive real numbers $\{\varepsilon_N\}_{N\geq 0}$, with $\varepsilon_N\rightarrow0$ as $N\rightarrow\infty$, and assume that the approximations $\{\u{}\}_{N\geq 0}$ from \eqref{eq: sequence approximations} satisfy \eqref{eq: finite residual epsilon}. Then, the weak limit $\uu\in \VV$ in \eqref{eq:weak* convergence approx} of the sequence $\{\u{}\}_{N\geq 0}$ is a critical point of $\E$, i.e., a weak solution of \eqref{Pb:main PDE}, and, in fact, $\u{}\to\uu$ strongly in $\VV$.
    \end{proposition}
    
    \begin{proof}
        Pick an arbitrary $v\in \VV$, and an appropriate sequence $\{v_N\}_{N\geq 0}$ converging strongly to $v$ as $N\to\infty$, and with $v_N\in \VV_N$ for each $N$. Similarly to the proof of Prop.~\ref{prop: u critical Galerkin}, and motivated by Asm.~\ref{ass:Galerkin density}, we apply the decomposition
        \begin{equation}\label{eq: u critical point E approx}
            \la \E'(\uu), v\ra = \la \E'(\u{}), v_N\ra + \la \E'(\u{}),v-v_N\ra + \la \E'(\uu)-\E'(\u{}),v\ra.
        \end{equation}
        Owing to~\eqref{eq:weak* convergence approx}, the third term on the right-hand side vanishes. Furthermore, due to the weak$*$-convergence, we note that the sequence $\{\E'(\u{})\}_{N\geq 0}$ is bounded in $\VV'$, and thus,
        \begin{equation*}
            |\la \E'(\u{}),v-v_N\ra| \leq \|\E'(\u{})\|_{\VV'}\|v-v_N\|_{\VV}\rightarrow 0,
        \end{equation*}
        as $N\rightarrow \infty$.
        Contrary to the proof of Prop.~\ref{prop: u critical Galerkin}, however, $\u{}$ is typically \textit{not} a weak solution of \eqref{Pb:weak form W}, and hence, we need to appropriately control the first term on the right-hand side of \eqref{eq: u critical point E approx}. To this end,
        using \eqref{eq: finite residual epsilon}, we observe that
        \begin{equation*}
           |\la \E'(\u{}), v_N\ra| \leq \|\E'(\u{})\|_{\VV'_N}\|v_N\|_{\VV}\leq \varepsilon_N\|v_N\|_{\VV}.
        \end{equation*}
        As the sequence $\{v_N\}_{N\geq 0}$ converges in $\VV$, it is bounded; therefore, exploiting that $\varepsilon_N\rightarrow 0$, we infer that
        $
            |\la \E'(\u{}), v_N\ra| \rightarrow 0$ as $N\rightarrow \infty$. Recalling~\eqref{eq: u critical point E approx}, we conclude that $\la\E'(\uu),v\ra=0$ for all $v\in\VV$, i.e., $\uu$ is a critical point of $\E$, and thus a solution of~\eqref{Pb:main PDE}.
            
    From Thm.~\ref{thm:beta} we obtain that $\uu \in \VV \cap \htto$, and we may therefore retrace the contradiction argument used in the proof of Cor.~\ref{cor:strong conv} for the sequence $\{\u{}\}_{N\geq 0}$. The only step that requires adaptation is the replacement of $\u{}$ by $I_h\uu$ in the duality product \eqref{eq: decomp gradients}, which, in the exact Galerkin case of $\un$, followed from~\eqref{eq:Galerkin orthogonality}; specifically, here, we need to estimate the second term on the right-hand side of the identity
    \[
    \la\E'(\uu)-\E'(\u{}),\uu-\u{}\ra
    =\la\E'(\uu)-\E'(\u{}),\uu-I_h\uu\ra
    -\la\E'(\u{}),I_h\uu-\u{}\ra.
    \]
    In light of~\eqref{eq: finite residual epsilon}, and by using the boundedness of $\{\u{}\}_{N\ge0}$ in $\VV$ as well as of the interpolation operator $I_h$, viz.
    \[
    \|I_h\uu\|_{\VV}
    \le\|\uu\|_{\VV}+\|\uu-I_h\uu\|_{\VV}\le C\|\uu\|_{\htto},
    \]
    see~\eqref{eq:graded estimate}, we have
    \begin{equation*}
        |\la \E'(\u{}), I_h\uu - \u{}\ra| \leq \varepsilon_N \|I_h\uu - \u{}\|_{\VV}\rightarrow 0
    \end{equation*}
    as $N\to\infty$, since $\varepsilon_N\rightarrow 0$. In particular, a weaker analog of Prop.~\ref{prop:an faster} holds for the approximations $\{\u{}\}_{N\geq 0}$, for which $\|\grad(\uu-\u{})\|_{\L^2(\Omega)}\rightarrow 0$ as $N\rightarrow \infty$, without a rate in terms of $\dim(\VV_N)$. This is sufficient for the contradiction argument of Cor.~\ref{cor:strong conv} to apply verbatim, and we conclude that there exists a subsequence of $\{\u{}\}_{N\geq 0}$ converging strongly in $\VV$.
    \end{proof}

    \begin{remark}\label{rmk:OCR approx}
        Within the setting of Thm.~\ref{thm:OCR Galerkin}, an approximate solution to~\eqref{Pb:main PDE} satisfying a prescribed error tolerance $\tau$ can be obtained in a finite number of operations. Specifically, let $\{\VV_N\}_{N\geq 0}$ be the subsequence of finite element spaces from Thm.~\ref{thm:OCR Galerkin}, and $\un\in\VV_N$, for subsequence indices~$N$, the associated Galerkin solutions, with a limit solution $u\in \VV\cap \htto$, such that~\eqref{eq:OCR general} holds true. Then, given a subsequence index~$N$ with $\dim(\VV_N)\gtrsim\tau^{-2}$, 
        consider the sequence
        $\{u^n_N\}_{n\geq 0}\subset \VV_N$ generated by~\eqref{eq:forms iteration scheme W} on $\WW=\VV_N$.
        \begin{enumerate}[(i)]
            \item In the general case, Thm.~\ref{thm: convergence in closed sbspc} guarantees that $\{u^n_N\}_{n\geq 0}$ admits a subsequence that converges strongly to \textit{some} Galerkin solution of~\eqref{eq:un}. If this limit coincides with the Galerkin solution $\un$ from the subsequence provided by Thm.~\ref{thm:OCR Galerkin}, then $\|\un - u^{n'}_N\|_{\Delta t}\to 0$, and consequently, there exists $n^\star=n^\star(N,\tau)\in\N$ such that $\|\un - u^{n'}_N\|_{\Delta t}\leq c\tau$ for all subsequence indices $n'\geq n^\star$, where $c>0$ is a constant independent of $\tau$ and of~$N$. Moreover, employing~\eqref{eq:OCR general}, the triangle inequality then gives
            \begin{equation}\label{eq:error tol remark}
                \|u - u^n_N\|_{\Delta t} 
                \lesssim \dim(\VV_N)^{-\nicefrac12}
                + \|\un - u^n_N\|_{\Delta t} \lesssim\tau \qquad \forall n'\geq n^\star.
            \end{equation}
            This shows that the overall error
            can be driven below $\mathcal{O}(\tau)$ by choosing
            $\dim(\VV_N)$ large enough, and by iterating the scheme~\eqref{eq:forms iteration scheme W} sufficiently often on~$\VV_N$, so that the two terms arising from the application of the triangle inequality above are appropriately balanced in terms of $\tau$.
            The iteration count $n^\star$ needed to reach this threshold is finite, though no explicit formula is available in the general case. We also stress that the choice of the initial guess on each space $\VV_N$ may play an important role in the iterative process as the iteration~\eqref{eq:forms iteration scheme W} could possibly converge to a Galerkin solution that differs from $\un$ if the final solution on $\VV_{N}$ is too far from $u^\star_{N-1}$, cf.~\eqref{eq:u0}.
            Similarly, the optimal rate for the residuals, viz.
            \[
            \|\E'(u^{n'}_N)\|_{\VV'}
            \lesssim\tau\qquad\forall n'\ge n^\star,
            \]
            follows the same lines of arguments as in  Cor.~\ref{cor:residual OCR},
            again under the same identification hypothesis.
    
            \item In the special case $\rho\leq \nicefrac{1}{\CP}$, the mapping $\mathsf{T}_{\Delta t}:\L^2(\Omega)\to \VV_N$ from~\eqref{def:op T} is a contraction with constant $0<r(\Delta t)<1$, cf.~Rmk.~\ref{rmk:special case}, so that the iteration process can be specified more explicitly: Indeed, the fixed point $\un$ is unique in this context, and the entire sequence $\{u^n_N\}_{n\geq 0}$ converges strongly to $\un$ for \textit{any} initial guess. Moreover, uniqueness ensures that the hierarchical embedding~\eqref{eq:u0b} 
            results in iterates converging to the \textit{correct} Galerkin solution on each level, so that~\eqref{eq:error tol remark} holds for every subsequence index $N$ with $\dim(\VV_N)\gtrsim\tau^{-2}$. In this regime, Banach's fixed point theorem provides the classical a priori error estimate
            \begin{equation*}
                \|u_N^\star-\u{}\|_{\Delta t} \leq r(\Delta t)^{n^\star}\|u_N^\star-u_N^0\|_{\Delta t},
            \end{equation*}
            with $r(\Delta t)$ from \eqref{eq:rt}, which results in the explicit iteration count
            \begin{equation}\label{eq:nstar remark}
                n^\star \gtrsim \frac{|\log(\tau)|}{|\log (r(\Delta t))|} \simeq\frac{\log(\dim(\VV_N))}{2|\log (r(\Delta t))|},
            \end{equation}
            showing that the number of iterations grows only logarithmically in $\dim(\VV_N)$, or equivalently in~$\tau^{-1}$.
        \end{enumerate}
    \end{remark}

\section{Numerical experiments}\label{sc:numerics}
    In this section, we present numerical experiments to investigate the optimal convergence rate (OCR) of the proposed iterative linearized finite-element method for the semilinear elliptic boundary value problem~\eqref{Pb:main PDE}. The computations aim to both validate our theoretical results derived in \S\ref{sc:ocr}--\S\ref{sc:approx} and to test the sharpness of our hypotheses by showing that convergence may fail if the hypotheses of Thm.~\ref{thm:OCR Galerkin} are not all satisfied.
    We solve the discrete problems using the IMEX iteration \eqref{eq:forms iteration scheme W}, starting from the zero initial guess $u^0\equiv 0\in\hoz$, with a sufficiently small time-step parameter $\Delta t>0$, cf.~Asm.~\ref{Ass:convergence}(i). All tests are based on the L-shaped domain
    \begin{equation}\label{eq:Lshaped}
    \Omega=(-1,1)^2\backslash([-1,0]\times[0,1]),
    \end{equation}
    which has a re-entrant corner at the origin in $\R^2$. For this domain, the principal eigenvalue $\lambda_1>0$ 
    of the Laplacian is numerically estimated (cf. \cite{Yuan2009}) by 
    \begin{equation*}
         \CP^{-1}\equiv \lambda_1= 9.639723\ldots,
    \end{equation*}
    with $\CP>0$ the Poincar\'e constant from \eqref{eq:Poincare L2}. 
    
    The initial (uniform) mesh comprises 12 triangles and 3 interior nodes; subsequent meshes are obtained by successive red-green-blue (RGB) refinements applied to a chosen subset of triangles so as to generate subdivisions that are appropriately graded toward the re-entrant corner. Motivated by~\eqref{eq:nstar remark}, on each mesh, we run the scheme \eqref{eq:forms iteration scheme W} for at most
        \begin{equation}\label{eq:gamman}
        n^\star=\lceil\gamma \log(\dim(\VV_N)) \rceil
        \end{equation}
    iterations, where $\gamma(\Delta t)\in \N$ is suitably chosen, and compute the approximate finite element solution~$\u{}\equiv \unn\in \VV_N$ as outlined in Rmk.~\ref{rmk:OCR approx}. All element integrals are performed using the standard 3-point quadrature rule at the midpoints of the triangle edges.

\subsection{Experiment 1: General case}\label{sc:experminent gen exp}
    In this first example, we aim to test the attainability of the OCR in the general case for an exponential nonlinearity. Specifically, we solve the following Dirichlet problem on the L-shaped domain $\Omega$ from~\eqref{eq:Lshaped}:
    \begin{alignat*}{2}
            -\Delta u &= 12\exp(-u^2) + g&\qquad&  \text{in } \Omega\\
            u&=0&& \text{on } \partial\Omega.
    \end{alignat*}
    We choose the right-hand side $g$ so that the exact solution is given by \begin{equation}\label{eq:u solution}
        u(x,y)=2r^{-\nicefrac{4}{3}}xy(1-x^2)(1-y^2),
    \end{equation}
    with the radial variable $r=(x^2+y^2)^{\nicefrac{1}{2}}$, which features a typical elliptic corner singularity at the origin $\mat c_1=(0,0)$; indeed, it can be verified that $u\in \htto$, for the weight function $\Phi_\beta=r^\beta$ with $\beta > \nicefrac{1}{3}$, which is in line with \eqref{eq:beta LB}. With $f(\cdot,s):=12\exp(-s^2) + g$, it is clear that $f(\cdot, s)\in \L^2(\Omega)$ for all $s\in \R$, and $f$ is continuously differentiable in its second argument. In fact, it is easy to see that
    \begin{equation*}
         |f_u(\x,s)| = \left|24se^{-s^2}\right| \leq 12\sqrt{\nicefrac{2}{e}}= 10.29\ldots \qquad \forall (\x,s)\in\Omega\times\R.
    \end{equation*}
    For every $\lambda>0$ we compute
    \begin{equation*}
        \sigma_f(\lambda) = \sup_{s\in \R}\left|\frac{1}{\lambda}-24se^{-s^2}\right| = \frac{1}{\lambda}+12\sqrt{\nicefrac{2}{e}},
    \end{equation*}
    and in turn
    \begin{equation*}
        \Lambda_f(\rho) = \left\{\lambda>0 \; : \; \nicefrac{1}{\lambda}+12\sqrt{\nicefrac{2}{e}} < \rho + \nicefrac{1}{\lambda}\right\},
    \end{equation*}
    which shows that $\Lambda_f(\rho)=(0,\infty)$ for $\rho>12\sqrt{\nicefrac{2}{e}}$. Moreover, we infer $\mu_f=0$ according to \eqref{eq:def mu}, yielding the bounds \eqref{eq:f' bounds}. Consequently, Asm.~\ref{Ass: f} is satisfied, and we find ourselves in the general case $\rho>\nicefrac{1}{\CP}$. Furthermore, the growth condition on $f_{uu}$ from Rmk.~\ref{rmk:aux pb existence} is clearly satisfied, and the dual problem \eqref{eq:dual pb} admits a solution.
    We choose $\Delta t=3$, as for $\rho\searrow 12\sqrt{\nicefrac{2}{e}}$ the largest admissible time-step satisfying Asm.~\ref{Ass:convergence}\eqref{Ass: delta_t bound} is given by
    \begin{equation*}
        \Delta t_{\text{sup}} = \frac{2}{12\sqrt{\nicefrac{2}{e}}-\nicefrac{1}{\CP}}= 3.06\ldots
    \end{equation*}
    In addition, recalling the definition of the energy functional $\E$ from~\eqref{eq:energy func}, we observe that
    \begin{equation*}
        |F(\x,t)|=\left|\int_0^t \left(12e^{-s^2} + g(\x)\right) \ds\right|\leq 6\sqrt{\pi} + tg(\x),
    \end{equation*}
    which implies that $\E(v)\rightarrow +\infty$ as $\|v\|_{\VV}\rightarrow \infty$, i.e., $\E$ is weakly coercive, and Asm.\ref{Ass:convergence}\eqref{Ass: E weak coercivity} is fulfilled.

    The re-entrant singularity of $u$ from \eqref{eq:u solution} necessitates a refinement toward $(0,0)$. To this end, we propose to use a graded mesh refinement (cf. \S\ref{sc:graded meshes}) with $\beta=0.4>\nicefrac{1}{3}$. The corresponding graded meshes are constructed based on appropriate RGB refinements of triangles close to~$\mat c_1$, with values $h\in\{0.25,0.15,0.08,0.035,0.016,0.008,0.0038,0.0019\}$; we present two examples in Fig.~\ref{Fig:Pb1 meshes}.
    By Rmk.~\ref{rmk:OCR approx} (see, in particular, the estimate~\eqref{eq:error tol remark}), we expect the sequence $\{\u{}\}_{N\geq 0}$ to converge (up to a subsequence) at optimal rate \eqref{eq:OCR}. 
    
    \begin{figure}
         \centering
        \includegraphics[scale=0.56]{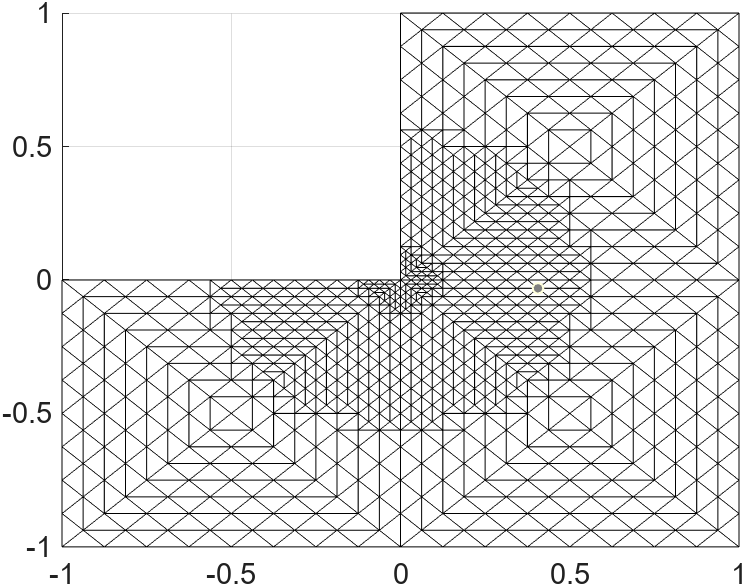}\hfill
         \includegraphics[scale=0.56]{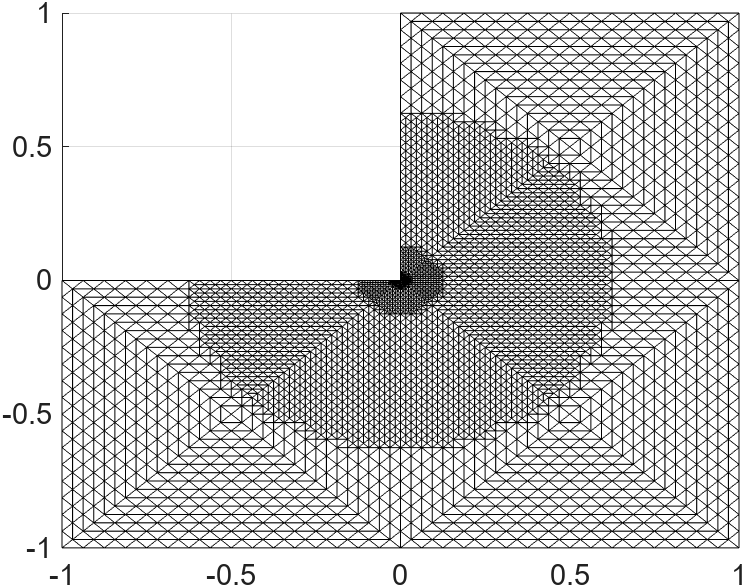}
         \caption{Graded meshes toward the origin in an L-shaped domain with $\beta=0.40$ for mesh size $h=0.035$ (left) and $h=0.016$ (right).}
         \label{Fig:Pb1 meshes}
    \end{figure}
    Fig.~\ref{Fig:Pb1 exponential} illustrates the global error between the exact solution $u$ and the iterative solution $\u{}$ against the number of degrees of freedom for each graded mesh. The iterative scheme \eqref{eq:forms iteration scheme W} is applied to each mesh until the (logarithmic) error slope between two consecutive meshes is less than $-0.49$, or the maximal number of iterations according to~\eqref{eq:gamman}
    is reached; here, $\gamma=1$ in \eqref{eq:gamman} suffices to obtain the desired convergence rate. Indeed, the convergence plot clearly shows that the expected convergence rate of $\dim(\VV_N)^{-\nicefrac{1}{2}}$ is achieved for the iterative linearized Galerkin scheme.

    \begin{figure}
    \begin{center}
        \includegraphics[scale=0.6]{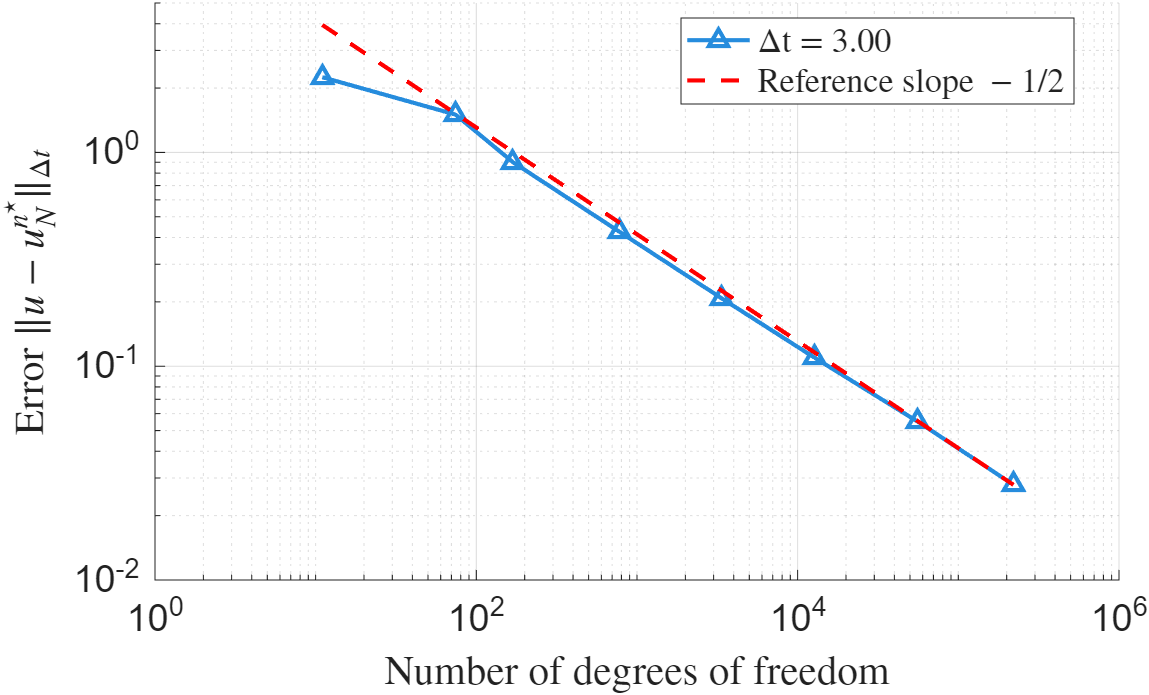}
        \caption{Experiment~1: Convergence behavior on graded meshes in the general case, using $\gamma=1$.}
        \label{Fig:Pb1 exponential}
    \end{center} 
    \end{figure}

\subsection{Experiment 2: Special case with discontinuous reaction derivative}\label{sc:experiment spec log}
    In our second experiment, we aim to showcase the
    redundancy of the continuity assumption on the reaction derivative $f_u$ in the special case. For this purpose, we consider the boundary value problem
    \begin{subequations}\label{Pb:exp 2}
    \begin{alignat}{2}
            -\Delta u &= \alpha\ln(1+|u-\nicefrac{1}{8}|) + g&\qquad&  \text{in } \Omega\label{Pb:exp 2a}\\
            u&=0&& \text{on } \partial\Omega,
    \end{alignat}
    \end{subequations}
    with a fixed parameter $0<\alpha<\nicefrac{1}{\CP}$. Here, we choose the right-hand side $g$ so that the exact solution remains given by \eqref{eq:u solution} from Exp.~\ref{sc:experminent gen exp}. By similar arguments as above, the right-hand side function, $f$, in~\eqref{Pb:exp 2a}
    satisfies Asm.~\ref{Ass: f} for $\rho>\alpha$ and $\mu_f=0$ since $\sigma_f(\lambda) = \alpha+\lambda^{-1}$ in~\eqref{eq: sigma(lambda)}, and $f_u(\cdot,s)$ is discontinuous at $s=\nicefrac{1}{8}$. Furthermore, choosing $\alpha <\rho \leq \nicefrac{1}{\CP}$, places us in the special case, for which $f_u$ is not required to be continuous, cf. \S\ref{ssc:OCR special case}. Referring to \S\ref{sc: convergence}, Asm.~\ref{Ass:convergence} is automatically fulfilled, and we employ the reduced iterative scheme \eqref{eq:reduced scheme} corresponding to $\Delta t =\infty$. As a result, Rmk.~\ref{rmk:OCR approx}(ii) ensures $\{\u{}\}_{N\geq 0}$ reaches the optimal rate of convergence.
    
    For the computations reported in Fig.~\ref{Fig:Pb2 special} the values $\alpha=9.6$ and $\gamma=1$ have been applied. Although $f_u$ is discontinuous when the solution takes the value $u=\nicefrac{1}{8}$, the numerical result exhibits no qualitative deviation from the smooth setting: the method attains the same convergence rate as observed in Exp.~\ref{sc:experminent gen exp}. In fact, the analytic machinery that was needed in the general case (e.g., uniform continuity of $f_u$ along connecting paths from Lem.~\ref{Lemma: int f' conv}) becomes unnecessary, because contraction directly enforces decay of the iteration error, cf. \S\ref{ssc:OCR special case}. Thus, Exp.~\ref{sc:experiment spec log} exhibits OCR in the special case despite the presence of derivative discontinuities.
    
    \begin{figure}
    \begin{center}
        \includegraphics[scale=0.5]{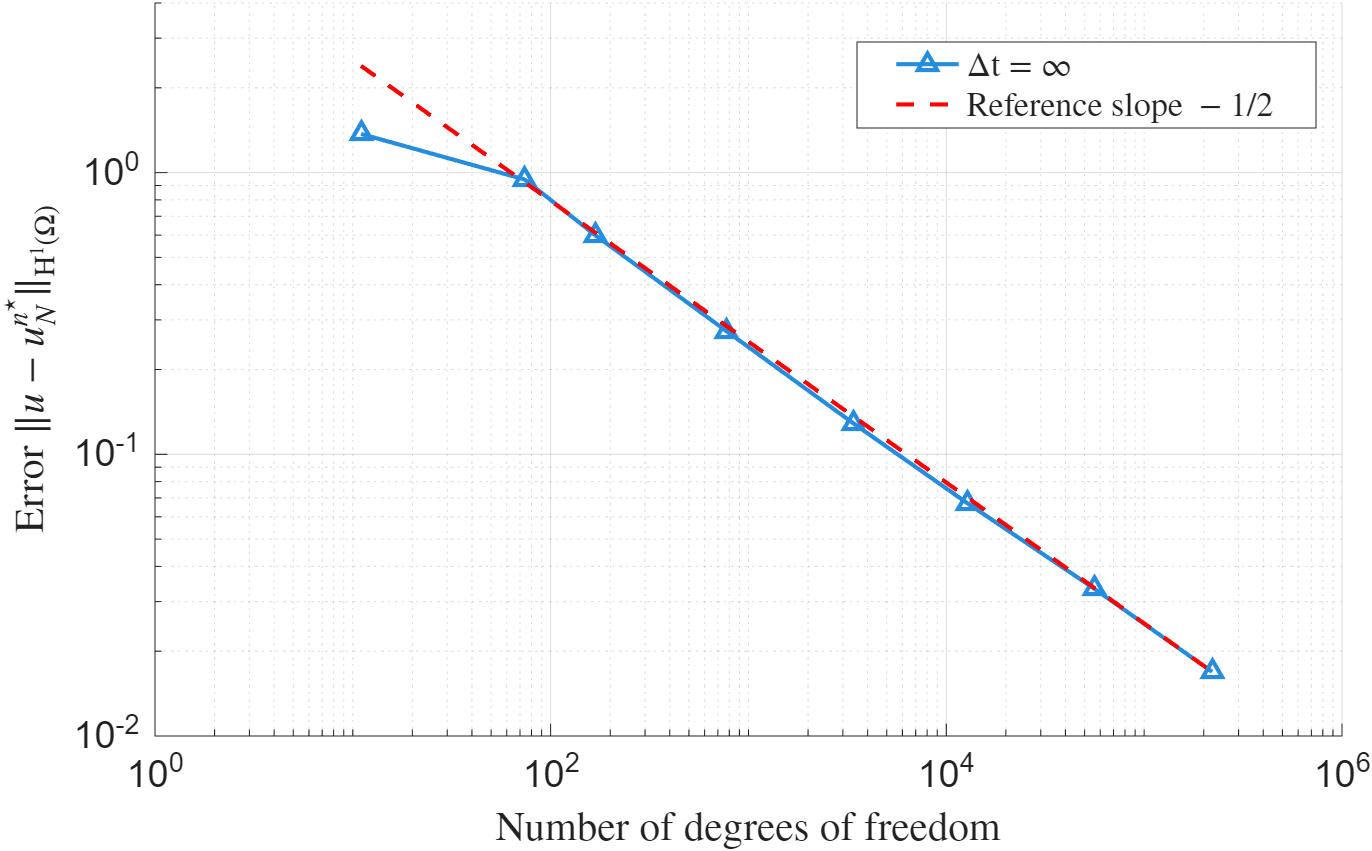}
        \caption{Experiment~2: Convergence behavior of the $\H^1$-error (i.e., w.r.t the norm~\eqref{eq:lambda norm} with $\lambda=1$) on graded meshes in the special case with reduced scheme \eqref{eq:reduced scheme}, using $\alpha=9.6$ and $\gamma=1$.}
        \label{Fig:Pb2 special}
    \end{center} 
    \end{figure}

\subsection{Experiment 3: Discontinuous reaction derivative}\label{sc:experiment gen div}
    This final experiment tests the sharpness of our framework's assumptions by exploring the general case with a discontinuous derivative $f_u$, demonstrating failure of optimal convergence. Specifically, we consider the same Dirichlet problem as in Exp.~\ref{sc:experiment spec log} with parameter $\alpha=20>\nicefrac{1}{\CP}$, i.e.,
    \begin{subequations}\label{Pb:exp 3}
    \begin{alignat}{2}
            -\Delta u &= 20\ln(1+|u-\nicefrac{1}{8}|) + g&\qquad&  \text{in } \Omega\label{Pb:exp 3a}\\
            u&=0&& \text{on } \partial\Omega,
    \end{alignat}
    \end{subequations}
    Again, we choose the right-hand side $g$ so that the exact solution is given by~\eqref{eq:u solution} from Exp.~\ref{sc:experminent gen exp}. Similarly to Exp.~\ref{sc:experiment spec log}, we note that the right-hand side, $f$, of~\eqref{Pb:exp 3a}
    satisfies Asm.~\ref{Ass: f} for $\rho>20$ and $\mu_f=0$. Thereby, the choice of the parameter $\alpha$ now results in a bound $\rho$ exceeding $\nicefrac{1}{\CP}$, taking the problem to the general case; notably, however, the continuity of $f_u$ in the second argument, as required by Thm.~\ref{thm:OCR Galerkin}, is violated at $u=\nicefrac18$.

    In the limit $\rho\searrow 20$ we have
    \begin{equation*}
        \Delta t_{\text{sup}} = \frac{2}{20-\nicefrac{1}{\CP}}= 0.19\ldots,
    \end{equation*}
    so $\Delta t =0.15$ satisfies Asm.~\ref{Ass:convergence}\eqref{Ass: delta_t bound}. Moreover, to verify Asm.~\ref{Ass:convergence}\eqref{Ass: E weak coercivity}, we observe the asymptotic behavior
    \[
    |F(\cdot,t)|=\left|\int_0^t f(\cdot,s)\,\ds\right|
    \le |t|\sup_{(\x,s)\in\Omega\times(0,t)}|f(\x,s)|
    \lesssim|t|\ln(1+|t|)\quad\text{as }|t|\to\infty;
    \]
    cf.~\eqref{eq:F}. Hence, the behavior of the energy from \eqref{eq:energy func} is dominated by its gradient term, and so $\E(v)\rightarrow +\infty$ as $\|v\|_{\VV}\rightarrow \infty$. 
    
    Figure \ref{Fig:Pb3 divergence} (left) illustrates the global error for the same sequence of graded meshes as in the previous experiments. Here, we choose $\gamma=10^3$ in~\eqref{eq:gamman}, which corresponds to $1.3\cdot 10^4$ iterations on the final mesh, and at most $6.4 \cdot 10^4$ iterations in total. Unlike the first two experiments, the sequence $\{\u{}\}_{N\geq 0}$ cannot achieve OCR within a reasonable number of iterations. It is interesting to observe that, for this experiment, the sequence of ratios $\nicefrac{a_N}{b_N}$ computed from \eqref{eq:def an bn} seems to stabilize at a positive value of $\limab\approx 0.12$ on finer meshes; see Fig.~\ref{Fig:Pb3 divergence} (right). Since the proof of Thm.~\ref{thm:OCR Galerkin} is based on the non-existence of the case $\limab>0$ (cf. \eqref{eq:def ell}), the occurring positivity of $\limab$ could be regarded as an a posteriori indicator that the hypotheses in Thm.~\ref{thm:OCR Galerkin} (albeit seemingly necessary) are violated. From this point of view, there are two possible reasons for this failure of OCR: Firstly, the existence of a solution to the dual problem \eqref{eq:dual pb} (or the Fredholm condition $\KK=\{0\}$, cf.~\S\ref{ssc:Aubin-Nitsche}) cannot be established in a straightforward way; furthermore, the weak coercivity of the dual energy, see Rmk.~\ref{rmk:aux pb existence}, does not seem to be obvious. In light of Thm.~\ref{thm:OCR Galerkin deg} we remark, however, that the non-existence of a dual solution is fairly unlikely. Secondly, since the reaction term $f$ is merely differentiable and its derivative is discontinuous, we conjecture that this lack of regularity is responsible for the absence of convergence in the present experiment.

    \begin{figure}
         \centering
         \includegraphics[scale=0.38]{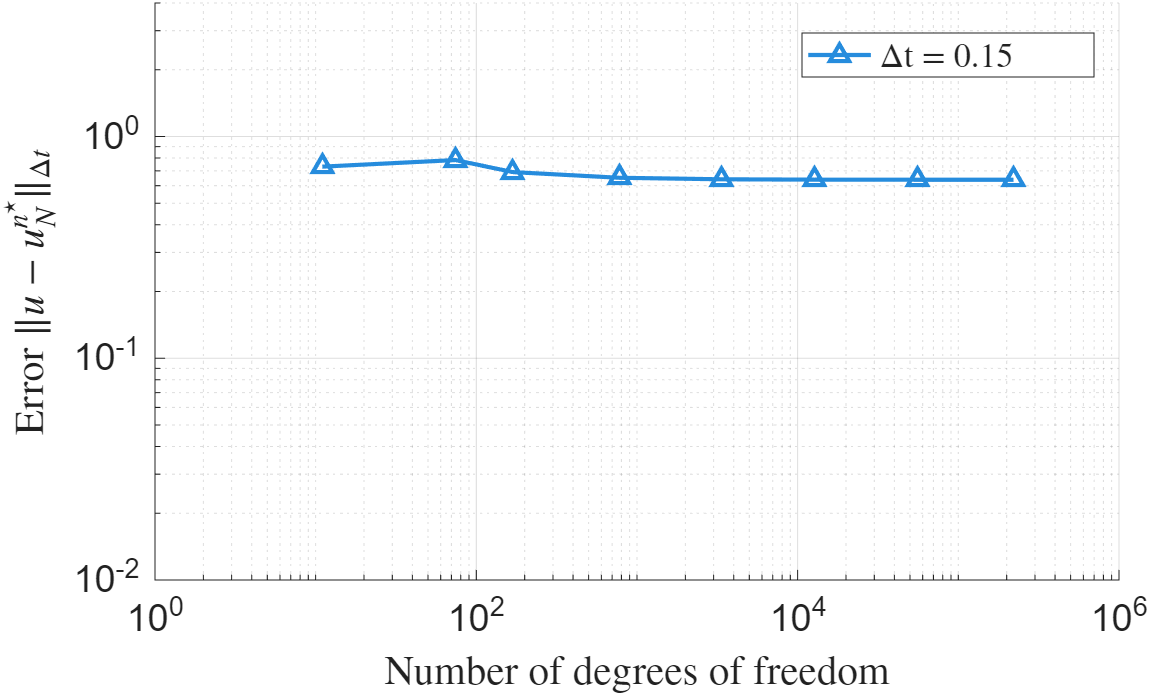}\hfill
         \includegraphics[scale=0.38]{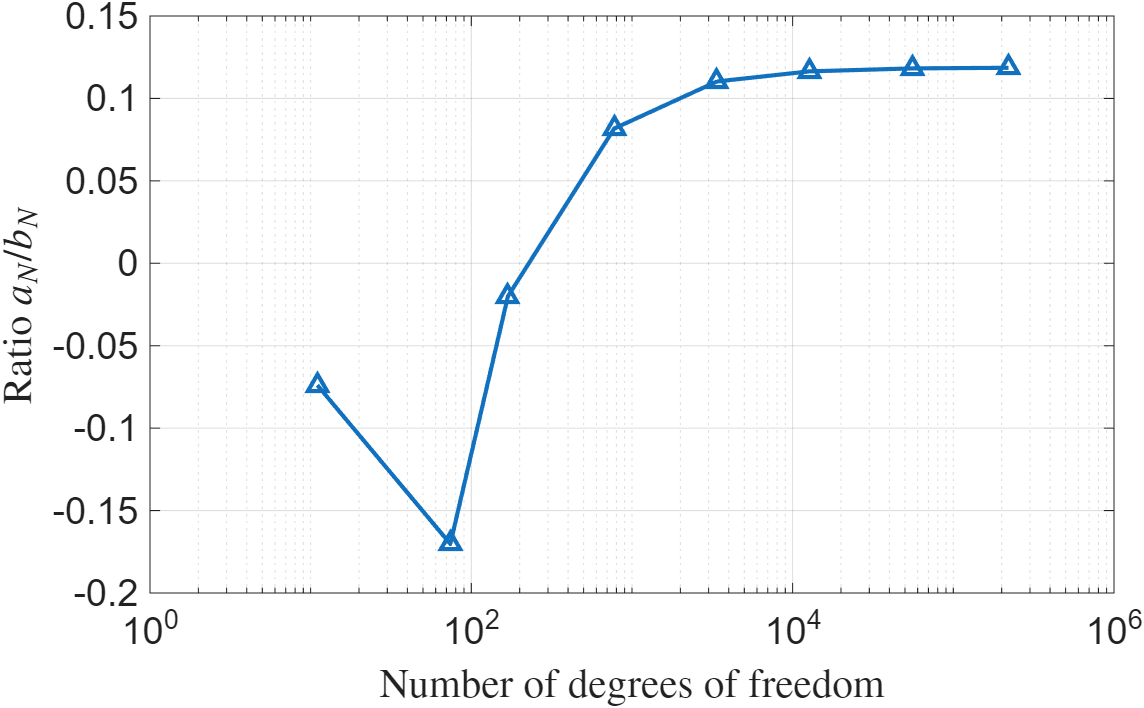}
         \caption{Experiment~3: Failure of convergence on graded meshes in the general case (left) and corresponding ratio $\nicefrac{a_N}{b_N}$ (right), using $\gamma=10^3$.}
        \label{Fig:Pb3 divergence}
    \end{figure}

\section{Conclusions}\label{sc:conclusion}
    In our work, we have developed an a priori error analysis of an IMEX-FEM framework for semilinear elliptic problems with non-monotone reaction terms that have uniformly bounded derivative. For such problems, we have demonstrated weak convergence of the Galerkin limits to a weak solution of \eqref{Pb:main PDE}, and we have established strong convergence for a subsequence. In the special case of a contractive iteration scheme, the entire sequence converges. The convergence analysis identifies two alternative regimes for the iteration---a contraction regime and an asymptotically monotone regime---and exploits this dichotomy in the error analysis. A central ingredient is Prop.~\ref{prop:an faster}, which provides an error bound expressed in terms of the nonlinearity and discrete residual quantities and, crucially, does not assume strong convergence of the Galerkin limits. As a result, strong convergence appears as a by-product of the estimate in the gradient-dominated case. We have then implemented this observation to force strong convergence in the general setting. Building on these convergence properties and on corner-weighted regularity, we have derived optimal convergence rates on suitably graded meshes. The same optimal rates have been shown to be attainable in practice with only a finite number of linearizations on each discrete space, with an explicit estimate of the order of the number of iteration steps in the contractive regime. Numerical experiments confirm the predicted rates, demonstrate the effectiveness of local mesh grading in the presence of corner singularities, and illustrate the sharpness of the proposed assumptions.

\bibliographystyle{alpha}
\bibliography{ref}

@book{Evans2010,
   author = {L.C. Evans},
   publisher = {American Mathematical Society},
   title = {Partial Differential Equations},
   volume = {19},
   year = {2010},
   edition = {2nd},
   series = {Graduate Studies in Mathematics}
}

@book{Schwab1998,
   author = {Schwab, C.},
   month = {10},
   pages = {154-167},
   publisher = {Clarendon Press},
   title={p- and hp- {F}inite {E}lement {M}ethods: {T}heory and {A}pplications in {S}olid and {F}luid {M}echanics},
   year = {1998},
   isbn={9780198503903},
   lccn={98023129},
  series={Numer. Math. Sci. Comput.},
}

@misc{Hardering2017,
      title={The {A}ubin--{N}itsche Trick for Semilinear Problems}, 
      author={H. Hardering},
      year={2017},
      howpublished={Tech. {R}eport 1707.00963, ar{X}iv.org},
      archivePrefix={arXiv},
      primaryClass={math.NA},
      url={https://arxiv.org/abs/1707.00963}, 
}

@article{Babuska1979,
   author = {I. Babuška and R.B. Kellogg and J. Pitkäranta},
   doi = {10.1007/BF01399326},
   issn = {0029-599X},
   issue = {4},
   journal = {Numer. Math.},
   month = {12},
   pages = {447-471},
   title = {Direct and inverse error estimates for finite elements with mesh refinements},
   volume = {33},
   year = {1979},
}

@article{Babuska1986,
    author = {Babuška, I. and Guo, B.Q.},
    title = {Regularity of the Solution of Elliptic Problems with Piecewise Analytic Data. {P}art {I}. {B}oundary Value Problems for Linear Elliptic Equation of Second Order},
    journal = {SIAM J. Math. Anal.},
    volume = {19},
    number = {1},
    pages = {172-203},
    year = {1988},
    doi = {10.1137/0519014}
}

@article{Yuan2009,
   author = {Q. Yuan and Z. He},
   doi = {10.1016/j.cam.2009.08.114},
   issn = {03770427},
   issue = {4},
   journal = {J. Comput. Appl. Math.},
   month = {12},
   pages = {1083-1090},
   title = {Bounds to eigenvalues of the {L}aplacian on {L}-shaped domain by variational methods},
   volume = {233},
   year = {2009},
}

@article{Wih,
    author = {Amrein, M. and Heid, P. and Wihler, T.P.},
    title = {A Numerical energy reduction approach for semilinear diffusion-reaction boundary value problems Based on steady-state iterations},
    journal = {SIAM J. Numer. Anal.},
    volume = {61},
    number = {2},
    pages = {755-783},
    year = {2023},
    doi = {10.1137/22M1478586}
}

@article{SpicherWihler2025,
   author = {F. Spicher and T.P. Wihler},
   doi = {10.1051/m2an/2025093},
   issn = {2822-7840},
   issue = {6},
   journal = {ESAIM: Math. Model. Numer. Anal.},
   month = {11},
   pages = {3363-3383},
   title = {Optimal finite element approximations of monotone semilinear elliptic pde with subcritical nonlinearities},
   volume = {59},
   year = {2025},
}

@article{kondrat1967boundary,
  title={Boundary problems for elliptic equations in domains with conical or angular points},
  author={Kondrat'ev, V.A.},
  year={1967},
  url={https://api.semanticscholar.org/CorpusID:122925591},
  journal = {Trudy Moskov. Mat. Obšč.},
  volume = {16},
  pages = {209-292}
}

@book{RenardyRogers,
   author = {M. Renardy and R.C. Rogers},
   city = {New York},
   doi = {10.1007/b97427},
   isbn = {0-387-00444-0},
   publisher = {Springer-Verlag},
   title = {An Introduction to Partial Differential Equations},
   volume = {13},
   year = {2004},
   edition = {2nd},
}

@article {Vexler:25,
    AUTHOR = {Vexler, B.},
     TITLE = {A priori error estimates for finite element discretization of
              semilinear elliptic equations with non-{L}ipschitz
              nonlinearities},
   JOURNAL = {ESAIM Math. Model. Numer. Anal.},
    VOLUME = {59},
      YEAR = {2025},
    NUMBER = {2},
     PAGES = {1095--1112},
      ISSN = {2822-7840,2804-7214},
       DOI = {10.1051/m2an/2025019},
       URL = {https://doi.org/10.1051/m2an/2025019},
}

@article {HHSW:26,
    AUTHOR = {He, Y. and Houston, P. and Schwab, C. and Wihler, T.P.},
     TITLE = {Exponential convergence of {\it hp}-{ILGFEM} for semilinear
              elliptic boundary value problems with monomial reaction},
   JOURNAL = {IMA J. Numer. Anal.},
    VOLUME = {46},
      YEAR = {2026},
    NUMBER = {3},
     PAGES = {1470--1493},
      ISSN = {0272-4979,1464-3642},
       DOI = {10.1093/imanum/draf030},
       URL = {https://doi.org/10.1093/imanum/draf030},
}

\end{document}